\documentclass[12pt]{amsart}
\usepackage{latexsym,amsfonts,amsmath,amssymb,epsfig,graphics}

\def\pen{\mathrm{pen}}

\def\argmin{\mathrm{argmin}}

\newtheorem{Th}{Theorem}

\newtheorem{Def}{Definition}
\newtheorem{Prop}{Proposition}
\newtheorem{Lemma}{Lemma}

\newtheorem{Cor}{Corollary}
\renewenvironment{proof}{\noindent{\bf Proof.}}{\hfill
  $\blacksquare$\par\noindent}
 \newcommand{\com}[1]{}

\newcommand{\E}{\ensuremath{\mathbb{E}}}
\newcommand{\D}{{\mathcal{ D}}}
\renewcommand{\P}{\ensuremath{\mathbb{P}}}

\newcommand{\R}{\ensuremath{\mathbb{R}}}

\renewcommand{\L}{\ensuremath{\mathbb{L}}}

\newcommand{\ga}{\ensuremath{\gamma}}
\newcommand{\al}{\ensuremath{\alpha}}
\newcommand{\de}{\ensuremath{\delta}}
\newcommand{\la}{\ensuremath{\lambda}}
\newcommand{\ka}{\ensuremath{\kappa}}

\newcommand{\p}{\ensuremath{\varphi}} 

\newcommand{\be}{\ensuremath{\beta}}

\newcommand{\Ne}{\ensuremath{\mathbb{N}}}
\renewcommand{\com}[1]{{\it #1}}

\newcommand{\Mcol}{\mathcal{M}}
\newcommand{\Mcols}{\widetilde{\mathcal{M}}}

\newcommand{\Mwealthcol}{\mathcal{M}^{(l)}}
\newcommand{\Mwealthcolprime}{\mathcal{M}'^{\:(l)}}

\newcommand{\Msievecol}{\mathcal{M}^{(s)}}
\newcommand{\Msievecolprime}{\mathcal{M}'^{\:(s)}}

\newcommand{\Mstratcol}{\mathcal{M}^{(h)}}

\newcommand{\Ws}{\mathbf{W}}

\numberwithin{equation}{section}

\title{Maxisets for model selection}

\author[]{F. Autin, E. Le pennec, J.M. Loubes and V. Rivoirard }

\address{{\bf{Florent Autin}} : Centre de Math\'ematiques et d'Informatique\\
39, rue F. Joliot Curie, 13453 Marseille Cedex 13.
\vskip .1in {\bf{Erwan Le Pennec}} : Laboratoire de Probabilit{\'e}s et Mod{\`e}les Al{\'e}atoires,
UMR 7599\\
Universit{\'e} Paris Diderot\\
175 rue du Chevaleret, 75013 Paris
  \vskip .1in
{\bf{Jean-Michel Loubes}} : Institut de Math\'ematiques de Toulouse,\\ Equipe de Probabilit\'es et de Statistique\\
Universit\'e de Toulouse Paul Sabatier,
118 Route de Narbonne, 31000 Toulouse. \vskip .1in
{\bf{Vincent Rivoirard}} : Laboratoire de Math{\'e}matiques, UMR 8628\\
Universit{\'e} Paris-Sud.\\
B{\^a}t 425, 91405 Orday cedex\\
And D\'epartement de Math\'ematiques et Applications UMR 8553\\
Ecole Normale Sup\'erieure \\
45, rue d'Ulm\\
75230 Paris Cedex 05
 \vspace{1cm}}

\email{autin@cmi.univ-mrs.fr}
\email{lepennec@math.jussieu.fr}
\email{Jean-Michel.Loubes@math.ups-tlse.fr}
\email{Vincent.Rivoirard@math.u-psud.fr}

\begin{document}
\begin{abstract}
We address the statistical issue of determining the maximal spaces
(maxisets)
where model selection procedures attain a given rate of convergence. By considering first general dictionaries, then orthonormal bases, we characterize these
maxisets in terms of approximation spaces. These results are
illustrated by  classical choices of wavelet model collections. For each of them, the maxisets are described in terms of functional spaces.
We take a special care of the issue of calculability and measure the
induced loss of performance in terms of maxisets.
\end{abstract}
\maketitle

\textbf{Keywords:} approximations spaces, approximation
theory, Besov spaces, estimation, maxiset, model selection,  rates of convergence.

\textbf{AMS MOS:} 62G05, 62G20, 41A25, 42C40.

\section{Introduction} \label{s:intro}
The topic of this paper lies on the frontier between statistics and
approximation theory. Our goal is to characterize the functions well
estimated by a special class of estimation procedures: the model
selection rules. Our purpose is not to build new model selection
estimators but to determine thoroughly the functions for which well known
model selection procedures achieve good performances. Of course,
approximation theory plays a crucial role in our
setting but surprisingly its role is
even more important than the one of statistical tools. 
This statement will be emphasized by the use of the
\emph{maxiset approach}, which  illustrates the well known fact that
``well estimating is well approximating''.

\noindent More precisely we consider
the classical Gaussian 
white noise model 
\begin{equation*}\label{model} 
dY_{n,t}=s(t)dt+\frac{1}{\sqrt{n}} dW_t,\quad t\in \D, 
\end{equation*} 
where $\D\subset\R$, $s$ is the unknown function, $W$ is the Brownian 
motion in $\R$ and  $n\in\mathbb{N}^*=\{1,2,\dots,\}$.
This model means that for
any $u\in\L_2(\D)$, 
\[
Y_n(u)=\int_\D u(t)dY_{n,t}=\int_\D u(t)
s(t)dt+\frac{1}{\sqrt{n}} W_u
\]
is  observable where $W_u=\int_\D u(t)dW_t$ is a centered Gaussian process such that for all functions $u$ and $u'$,
\[
\E[W_u W_{u'}]= \int_\D  u(t)  u'(t) dt.
\]
We take a noise level of the form
$1/\sqrt{n}$ to refer to the asymptotic equivalence between the
Gaussian white noise model and the classical regression model with $n$
equispaced observations (see \cite{nussbaum}). 

\noindent Two questions naturally arise: how to construct an estimator $\hat s$
of $s$ based on the observation $dY_{n,t}$ and how to measure its
performance? Many estimators have been proposed in this setting
(wavelet thresholding, kernel rules, Bayesian procedures...). In this
paper, we only focus on
model selection techniques  described accurately in the next paragraph.
\subsection{Model selection procedures}\label{s:model}

The model selection methodology consists in constructing an estimator
by minimizing an empirical contrast $\ga_n$ over a given set, called a
model. 
The pioneer
work in model selection goes back in the 1970's with Mallows \cite{mal}
and Akaike \cite{aka}. Birg\'e and Massart develop the whole modern theory of
model selection in \cite{MR2288064,MR1848946,MR1848840} or
\cite{MR1679028} for instance. Estimation of a regression function
with model selection estimators is considered by Baraud in
\cite{Baraud2,Baraud1}, while inverse problems are tackled by
Loubes and Lude\~na \cite{lou1,lou2}. Finally model selection techniques provide nowadays
valuable tools in statistical learning (see  Boucheron et al. \cite{bou}).

\noindent In nonparametric estimation, performances of estimators are usually measured by using the quadratic norm, which gives rise to the following empirical quadratic contrast
\[
\ga_n(u)=-2Y_n(u)+\|u\|^2
\]
for any function
$u$, where $\|\cdot\|$ denotes the norm associated to $\L_2(\mathcal{D})$.
We assume that we are  given a dictionary of functions  of $\L_2(\D)$,
denoted by $\Phi=(\p_i)_{i\in\mathcal{I}}$ where $\mathcal{I}$ is a
countable set and we consider $\mathcal{M}_n$, a collection of models
spanned by some functions of $\Phi$. For any $m\in \mathcal{M}_n$, we denote by $\mathcal{I}_m$ the subset of $\mathcal{I}$ such that
\[m=\mbox{span}\{\p_i: \quad i\in \mathcal{I}_m\}\]
and $D_m\leq|\mathcal{I}_m|$ the
dimension of $m$.
Let $\hat s_m$ be the function that minimizes the
 quadratic empirical criterion  $\ga_n(u)$  with respect to $u\in m$. 
A straightforward computation shows that the estimator $\hat{s}_m$ is the projection of the data onto the space
$m$. So, if $\{e_1^m,\dots,e_{D_m}^m\}$ is an orthonormal basis (not
necessarily related to $\Phi$) of $m$
and
\[\hat\be_i^m=Y_n(e_i^m)=\int_\mathcal{D}e_i^m(t)dY_{n,t}\]
 then
\[
\hat s_m=\sum_{i\in \mathcal{I}_m}\hat\be_i^me_i^m,\quad\text{and}\quad \ga_n(\hat s_m)=-\sum_{i\in
\mathcal{I}_m}(\hat\be_i^m)^2.
\]
Now, the issue is the selection of the best model $\hat m$ from the data which 
gives rise to the \emph{model selection estimator} $\hat{s}_{\hat{m}}$.  For this purpose, a penalized rule is
considered, which aims at selecting an estimator, close enough to the
data, but still lying in a small space to avoid overfitting
issues. 
Let ${\rm pen}_n(m)$ be a penalty function
which increases when $D_m$ increases.
The model $\hat
m$ is selected using the following penalized criterion
\begin{equation}
\label{eq:minimization}
\hat m=\arg\min_{m\in \mathcal{ M}_n}\left\{\ga_n(\hat s_m)+\pen_n(m)\right\}.
\end{equation}
The choice of the model collection and the associated penalty are then
the key issues handled by model selection theory. We point out that the choices of both the model collection and the penalty function  should depend on the noise level. This is emphasized by the subscript $n$ for $\Mcol_n$ and $\pen_n(m)$.

\noindent The asymptotic behavior of model selection estimators has
been studied by many authors. We refer to 
Massart \cite{mas} for general references and 
recall hereafter the main oracle type inequality. 
Such an oracle inequality provides a non asymptotic control on the estimation error with 
respect to a bias term $\|s-s_m\|$, where $s_m$ stands for the best
approximation  (in the $\mathbb{L}_2$ sense) of the function  $s$ by
a function of $m.$ In other words $s_m$ is the orthogonal projection of
$s$ onto $m$, defined by
\[
s_m=\sum_{i\in  \mathcal{I}_m}\be_i^m e_i^m,\quad \be_i^m=\int_\mathcal{D}e_i^m(t)s(t)dt.
\]
\begin{Th}[Theorem 4.2 of \cite{mas}]\label{massart}
Let $n\in \mathbb{N}^\star$ be fixed and let $(x_m)_{m\in \mathcal{ M}_n}$ be some family of positive numbers such that
\begin{equation}\label{kraft} \quad \quad \quad 
\sum_{m\in \mathcal{ M}_n}\exp(-x_m)=\Sigma_n<\infty.
\end{equation}
Let $\kappa>1$ and assume that
\begin{equation}\label{minpen}
\pen_n(m)\geq \frac{\kappa}{n}\left(\sqrt{D_m}+\sqrt{2x_m}\right)^2.
\end{equation}
Then,  almost surely, there  exists some  minimizer $\hat  m$ of  the penalized
least-squares criterion
\[
\ga_n(\hat s_m)+\pen_n(m)
\]
over  $m\in \mathcal{  M}_n$.  Moreover, the  corresponding penalized  least-squares
estimator $\hat s_{\hat m}$ is unique and the following inequality is valid:
\begin{equation}\label{oracletypeMas}
\E\left[\|\hat   s_{\hat   m}-s\|^2\right]\leq  C\left[\inf_{m\in   \mathcal{
M}_n}\left\{\|s_m-s\|^2+\pen_n(m)\right\}+\frac{1+\Sigma_n}{n}\right],
\end{equation}
where $C$ depends only on $\kappa$.
\end{Th}
\noindent Equation (\ref{oracletypeMas}) is the key result to establish optimality of penalized estimators under oracle or minimax points of view. In this paper, we focus on an alternative to these approaches: the maxiset point of view.
\subsection{The maxiset point of view}\label{s:maxiset}
Before describing the maxiset approach, let us briefly recall that for a given procedure $s^*=(s_n^*)_n$, the minimax study of $s^*$ consists in comparing the rate of convergence of $s^*$ achieved on a given functional space
$\mathcal{F}$ with the best
possible rate achieved by any estimator. More precisely,
let
$\mathcal{ F}(R)$  be the ball of radius
$R$ associated with $\mathcal{ F}$, the procedure $s^*=(s_n^*)_n$ achieves the rate $\rho^*=(\rho^*_n)_n$ on $\mathcal{ F}(R)$ if 
\[
\sup_n\left\{(\rho^*_n)^{-2}\sup_{s\in                                   \mathcal{
F}(R)}\E\left[\|s_n^*-s\|^2\right]\right\}<\infty.
\]
 To check that a procedure is optimal from the minimax
point of view (said to be minimax), it must be proved that its rate of convergence achieves
the best rate among any procedure on each ball of the class. 
This minimax approach is extensively used and many methods cited above are proved to
be minimax in different statistical frameworks.

\noindent However, the choice of the function class is subjective and, in the
minimax framework, statisticians have no idea whether there are other functions well
estimated at the rate $\rho^*$ by their procedure. A different point
of view is to consider the procedure  $s^*$  as given and search  all the
functions $s$ that are well estimated at a given rate $\rho^*$: this is the
\emph{maxiset} approach, which has been proposed by Kerkyacharian and Picard \cite{kp-a}.
The maximal space, or maxiset, of the procedure $s^*$ for this rate 
$\rho^*$ is defined as the set of all these functions. Obviously,
the larger the maxiset, the better the procedure. We
set the following definition. 
\begin{Def}
Let $\rho^*=(\rho^*_n)_n$ be a decreasing sequence of
positive real numbers and let $s^*=(s_n^*)_n$ be an
estimation procedure. The maxiset of $s^*$  associated with the
rate $\rho^*$  is 
\[
MS(s^*,\rho^*) =\left\{s\in\L_2(\D):\quad
\sup_n\left\{(\rho^*_n)^{-2}\E\left[\|s_n^*-s\|^2\right]\right\}<\infty\right\},
\]
the ball of radius $R>0$ of the maxiset is defined by
\[
MS(s^*,\rho^*)(R) =
\left\{s\in\L_2(\D):\quad
\sup_n\left\{(\rho^*_n)^{-2}\E\left[\|s_n^*-s\|^2\right]\right\}\leq R^2\right\}.
\]
\end{Def}
\noindent Of course,  there exist  connections between maxiset  and minimax  points of
view: $s^*$ achieves the rate
$\rho^*$ on $\mathcal{ F}$ if and only if 
\[
\mathcal{ F}\subset MS(s^*,\rho^*).
\]
In the white noise setting, the maxiset theory has been investigated
 for a wide range of estimation procedures, including kernel,
 thresholding and Lepski procedures, Bayesian or linear
 rules. We refer to \cite{aut-b},
 \cite{apr}, \cite{br}, \cite{cdkp}, \cite{kp-a},  \cite{riv-a}, and \cite{riv-b} for general
 results. Maxisets have also been investigated for other statistical
 models, see  \cite{aut-a} and \cite{rt}. 



\subsection{Overview of the paper}
The goal of this paper is to investigate maxisets of model selection procedures. Following the classical model selection literature,  we only use penalties proportional to the dimension $D_m$
of $m$:
\begin{equation}\label{eq:pen}
\pen_n(m) = \frac{\lambda_n}{n} D_m,
\end{equation}
with $\lambda_n$ to be specified. Our main result characterizes these maxisets in terms of approximation
spaces. More precisely, we establish an equivalence between the
statistical performance of  $\hat s_{\hat m}$ and the approximation
properties of the model collections $\mathcal{M}_n$. With
\begin{equation}\label{rateL}
\rho_{n,\al}=\left(\frac{\la_n}{n}\right)^{\frac{\al}{1+2\al}}
\end{equation}
for any $\alpha>0$, Theorem ~\ref{main}, combined with Theorem~\ref{massart} proves that,
for a given function $s$,
the quadratic risk $\E[\|s-\hat s_{\hat m}\|^2]$ decays at the rate $\rho_{n,\al}^2$ 
if and only if
the deterministic quantity 
\begin{equation}\label{Q}
Q(s,n)=\inf_{m\in
    \mathcal{ M}_n}  \left\{\|    s_{m}-s\|^2+ \frac{\lambda_n}{n} D_m
\right\}
\end{equation}
 decays
at the rate $\rho_{n,\al}^2$ as well. This
result holds with mild assumptions on $\lambda_n$ and under an
embedding assumption on the model collections ($\mathcal{M}_n
\subset \mathcal{M}_{n+1}$). Once we impose additional structure on
the model collections, the deterministic condition can be rephrased as
a linear approximation property and a non linear one as stated in Theorem~\ref{main2}.\\
We illustrate these results for three different model collections  based on wavelet bases. The
first one deals with sieves in which all the models are embedded, the
second one with the collection of all subspaces spanned by vectors of
a given basis. For these examples, we handle the issue of calculability
and give  explicit characterizations of the maxisets. In the third example, we provide an
intermediate choice of model collections and use the fact that
the embedding condition on the model collections can be relaxed. Finally performances of these estimators are compared and discussed.

\noindent The paper is organized as follows.  Section~\ref{s:main} describes the main general results established in this paper. More precisely, we specify results valid for general dictionaries in Section \ref{dico}. In Section \ref{basis}, we focus on the case where $\Phi$ is an orthonormal family.  Section~\ref{part3} is devoted to the
illustrations of these results for some model selection estimators associated with wavelet methods. In particular, a comparison of  maxiset performances are provided and discussed. Section~\ref{proofs} gives the proofs of our results.
\section{Main results} \label{s:main}
As explained in the introduction, our  goal is  to
investigate  maxisets associated with model  selection estimators $\hat{s}_{\hat{m}}$ 
where the penalty function is defined in (\ref{eq:pen})
and with the rate $\rho_\al=(\rho_{n,\al})_n$ where $\rho_{n,\al}$ is specified in (\ref{rateL}). 
Observe  that $\rho_{n,\al}$  depends on the choice of $\lambda_n$.  It can  be for instance polynomial, or can take the classical form \[\rho_{n,\al}=\left(\frac{\log n}{n}\right)^{\frac{\al}{1+2\al}}.\]
So we wish to determine
\[
MS(\hat{s}_{\hat{m}},\rho_\al)=\left\{s\in\L_2(\D):\quad
\sup_n\left\{\rho_{n,\al}^{-2}\E \left[\|\hat s_{\hat m}-s\|^2\right]\right\}<\infty\right\}.
\]
In the sequel, we use the following notation: if $\mathcal{ F}$ is a given space \[MS(\hat{s}_{\hat{m}},\rho_\al):=:\mathcal{ F}\]
means  that for any $R>0$, there exists $R'>0$ such that 
\begin{equation}\label{inc1}
MS(\hat{s}_{\hat{m}},\rho_\al)(R)\subset \mathcal{ F}(R')
\end{equation}
and  for  any $R'>0$,  there  exists  $R>0$  such that  
\begin{equation}\label{inc2}
\mathcal{  F}(R')\subset MS(\hat{s}_{\hat{m}},\rho_\al)(R).
\end{equation}
\subsection{The case of general dictionaries}\label{dico}
In this section, we make no assumption on $\Phi$. Theorem~\ref{massart} is a non asymptotic result while
maxisets results deal with rates of convergence (with asymptotics in
$n$). Therefore obtaining maxiset results for model selection
estimators requires a structure on the sequence of model collections. We first focus on the case of nested model collections ($\Mcol_n \subset \Mcol_{n+1}$). Note that this does not imply a strong structure on the model collection  for a given $n$. In particular, this does not imply that the models are nested.
Identifying the maxiset $MS(\hat{s}_{\hat{m}},\rho_\al)$ 
is a two-step procedure. 
We need to establish inclusion (\ref{inc1}) and inclusion (\ref{inc2}). Recall that we have introduced previously
\[
Q(s,n) =\inf_{m\in
    \mathcal{ M}_n}  \left\{\|    s_{m}-s\|^2+ \frac{\lambda_n}{n} D_m
    \right\}.
\] 
Roughly speaking, Theorem~\ref{massart} established by Massart proves
that any function $s$ satisfying
\[
\sup_n\left\{\rho_{n,\al}^{-2}Q(s,n)\right\}\leq (R')^2\]
belongs to the maxiset $MS(\hat{s}_{\hat{m}},\rho_\al)$ and thus
provides inclusion  (\ref{inc2}).
The following theorem establishes inclusion (\ref{inc1}) and
highlights  that $Q(s,n)$ 
plays a capital role.
\begin{Th}\label{main}
Let $0<\al_0<\infty$ be fixed. Let us assume that the sequence of model collections
satisfies for any $n$
\begin{equation}\label{hyplambdaembed}
\Mcol_n \subset \Mcol_{n+1},
\end{equation}
and that  the sequence of positive numbers
$(\la_n)_n$ is non-decreasing and satisfies 
\begin{equation}\label{hyplambdalim}
\lim_{n\to +\infty}n^{-1}\la_n=0,
\end{equation}
and there exist $n_0\in\Ne^*$ and two constants $0<\delta\leq \frac{1}{2}$ and $0<p<1$ such that for $n\geq n_0$,
\begin{equation}\label{hyplambda}
 \la_{2n}\leq 2(1-\delta)\la_{n},
\end{equation}
\begin{align}\label{eq:Kraftmain}
\sum_{m\in\Mcol_n} e^{-\frac{(\sqrt{\lambda_n}-1)^2D_m}{2}} \leq \sqrt{1 - p}
\end{align}
and 
\begin{equation}\label{hyplambda0}
\la_{n_0}\geq \Upsilon(\de,p,\al_0),
\end{equation}
where $\Upsilon(\de,p,\al_0)$ is a positive constant only depending on $\al_0$, $p$ and $\de$ defined in Equation (\ref{erwan}) of Section 4.
Then, the penalized  rule  $\hat s_{\hat  m}$ is such that
for any $\al\in  (0,\al_0]$, for any $R>0$, there exists  $R'>0$ such that for
$s\in\L_2(\D)$,
\[\sup_n\left\{\rho_{n,\al}^{-2}\E\left[\|\hat    s_{\hat
m}-s\|^2\right]\right\}\leq R^2  
\Rightarrow 
\sup_n\left\{\rho_{n,\al}^{-2}Q(s,n)\right\}\leq (R')^2.\]
\end{Th}
\noindent Technical Assumptions (\ref{hyplambdalim}), (\ref{hyplambda}), (\ref{eq:Kraftmain}) and (\ref{hyplambda0}) are very mild and could
be partly relaxed while preserving the results. Assumption
(\ref{hyplambdalim}) is necessary to deal with rates converging to 0.
Note that the classical cases
$\la_n=\la_0$   or    $\la_n=\la_0\log(n)$   satisfy
(\ref{hyplambdalim}) and   (\ref{hyplambda}). Furthermore, Assumption
(\ref{hyplambda0})  is always satisfied when $\lambda_n = \lambda_0
\log(n)$ or when $\lambda_n=\lambda_0$ with $\lambda_0$ large enough.
Assumption (\ref{eq:Kraftmain}) is very close to
Assumptions~\eqref{kraft}-(\ref{minpen}). In particular, if there exist two constants $\ka>1$ and $0<p<1$ such that for any $n$,
\begin{align}\label{eq:Kraftmainbis}
\sum_{m\in\Mcol_n} e^{-\frac{(\sqrt{\ka^{-1}\lambda_n}-1)^2D_m}{2}} \leq \sqrt{1 - p}
\end{align}
then, since 
\[\pen_n(m) = \frac{\lambda_n}{n} D_m,\]
Conditions (\ref{kraft}), (\ref{minpen}) and (\ref{eq:Kraftmain}) are all satisfied.
The  assumption $\al\in (0,\al_0]$ can  be relaxed
for particular model collections, which will be
highlighted in  Proposition~\ref{nestedtheo} of Section \ref{nested}. Finally, Assumption~(\ref{hyplambdaembed}) can be removed for some special
choice of model collection $\Mcol_n$ at the price of a slight
overpenalization as it shall be shown in
Proposition~\ref{prop:approxmod} and Section~\ref{special}.

\noindent Combining Theorems \ref{massart} and \ref{main} 
gives a first
characterization of the maxiset of the model selection procedure
$\hat{s}_{\hat m}$:
\begin{Cor}\label{corollaire}
Let $\al_0<\infty$ be fixed. 
Assume that Assumptions (\ref{hyplambdaembed}), (\ref{hyplambdalim}), (\ref{hyplambda})
~(\ref{hyplambda0}) and~(\ref{eq:Kraftmainbis}) are satisfied.
Then for any $\al\in  (0,\al_0]$,
\[MS(\hat s_{\hat m},\rho_\al):=:\left\{s\in \L_2(\D): \quad \sup_n\left\{\rho_{n,\al}^{-2}Q(s,n)\right\}<\infty\right\}.\] 
\end{Cor}
\noindent The maxiset of $\hat s_{\hat m}$ is characterized by a deterministic
approximation property
of $s$ with respect to the models $\Mcol_n$. It can be related
to some classical approximation properties of $s$ in terms of approximation rates if the functions of $\Phi$ are orthonormal. 
\subsection{The case of orthonormal bases}\label{basis}
 From now on, $\Phi=\{\p_i\}_{i\in\mathcal{I}}$ is assumed to be an
orthonormal basis (for the $\L_2$ scalar product). We also assume that the model collections
$\Mcol_n$ are constructed through restrictions of a single model
collection $\Mcol$. Namely,  given a collection of models $\Mcol$ we
introduce a sequence $\mathcal{J}_n$ of increasing subsets of the indices
set $\mathcal{I}$ and we 
 define the intermediate collection $\Mcol'_n$ as
\begin{equation}\label{eq:structureaux}
\Mcol'_n = \{ m'=\mbox{span}\{\p_i:\quad i\in {\mathcal I}_m\cap  {\mathcal J}_n\}:\quad
 m \in \Mcol\}.
\end{equation}
The model collections $\Mcol'_n$ do not necessarily
satisfy the embedding condition~(\ref{hyplambdaembed}). Thus, we define
\[
\Mcol_n = 
\bigcup_{k\leq n} \Mcol_k'
\]
so $\Mcol_n \subset
\Mcol_{n+1}$.  The assumptions on $\Phi$ and on the model collections allow to give an explicit characterization of the maxisets. We denote $\Mcols = \cup_n \Mcol_n = \cup_n
\Mcol_n'$. Remark that without any further assumption 
$\Mcols$ can be a larger model collection than $\Mcol$. Now, let us denote by $V=(V_n)_n$ the sequence of approximation spaces defined by
\[
V_n = \text{span}\{ \p_i:\quad i \in \mathcal{J}_n \}
\]
and consider the corresponding approximation space
\[
\mathcal{L}_V^{\al}= \left\{s\in\L_2(\D):\quad 
\sup_n \left\{\rho_{n,\al}^{-1}
  \|P_{V_n}s - s\|\right\}<\infty
\right\},
\]
where $P_{V_n}s$ is the projection of $s$ onto $V_n$.
Define also another kind of approximation sets: 
\[
\mathcal{ A}_{_{\Mcols}}^{\al}= \left\{s\in\L_2(\D):\quad \sup_{M>0}\left\{
M^{\al}\inf_{\{m\in \Mcols :\: D_m\leq
M\}}\|s_m-s\|\right\}<\infty\right\}.
\]
The corresponding balls of radius $R>0$ are defined, as usual, by replacing $\infty$
by $R$ in the previous definitions. 
We have the following result.
\begin{Th}\label{main2}
Let $\al_0<\infty$ be fixed. 
Assume that  (\ref{hyplambdalim}), (\ref{hyplambda}),
~(\ref{hyplambda0}) and~(\ref{eq:Kraftmainbis}) are satisfied.
Then,  the penalized  rule  $\hat s_{\hat  m}$
satisfies the following result:
for any $\al\in~(0,\al_0]$,
\[
MS(\hat s_{\hat m},\rho_\al):=:\mathcal{ A}_{_{\Mcols}}^{\al} \cap
\mathcal{L}_V^{\al}.
\]
\end{Th}

\noindent The result pointed out in Theorem \ref{main2} links the
performance of the estimator to an approximation property for the
estimated function.
This approximation property is decomposed into a linear approximation measured by $\mathcal{L}_V^{\al}$ and a  non linear approximation measured by 
$\mathcal{ A}_{_{\Mcols}}^{\al}$. The linear condition is due to the use of the reduced model
collection $\Mcol_n$ instead of $\Mcol$, which is often necessary to
ensure either
the calculability of the estimator or Condition~(\ref{eq:Kraftmainbis}). It plays the role of a minimum regularity
property that is easily satisfied.

\noindent Observe that if we have one model collection, that is for any $k$ and $k'$, $\Mcol_k=\Mcol_{k'}=\Mcol$, $\mathcal{J}_n=
\mathcal{I}$ for any $n$ and thus $\Mcols=\Mcol$. Then \[\mathcal{L}_V^{\al}=\mbox{span}\left\{\p_i:\quad i\in\mathcal{I}\right\}\]
and Theorem \ref{main2} gives
\[MS(\hat s_{\hat m},\rho_\al):=:\mathcal{ A}_{_{\Mcol}}^{\al}.\]

\noindent The spaces $\mathcal{ A}_{_{\Mcols}}^{\al}$ and $ \mathcal{L}_V^{\al}$ highly depend on the models and the approximation space. At first glance, the best choice seems to be $V_n= \L_2(\mathcal{D})$ and\[\mathcal{M} = \{ m:\quad \mathcal{I}_m \subset\mathcal{I}\}\quad\] since the infimum in the definition of $\mathcal{  A}_{_{\Mcols}}^{\al}$ becomes smaller when the collection is enriched.
There is however a price to pay when enlarging the model collection:
the penalty has to be larger to satisfy (\ref{eq:Kraftmainbis}),  which deteriorates the convergence rate.
A second issue comes from the tractability of the minimization~(\ref{eq:minimization})
itself
 which
will further limit the size of the model collection.

\noindent To avoid considering the union of $\Mcol'_k$, that can
dramatically increase
the number of models considered for a fixed $n$, leading to large
penalties, 
we can relax the
assumption that the penalty is proportional to the dimension.  
Namely, for any $n$, for any $m\in \Mcol_n'$, there exists $\tilde m\in 
\Mcol$ such that
\[m=\mbox{span}\left\{\p_i:\quad i\in \mathcal{I}_{\tilde m}\cap \mathcal{J}_n\right\}.\]
Then for any model $m\in\Mcol_n'$, we replace the
dimension $D_m$  by the larger dimension $D_{\tilde m}$ and we set
\[
\widetilde\pen_n(m)= \frac{\la_n}{n} D_{\tilde m}.
\]
The minimization of the corresponding penalized criterion over all
model in $\Mcol_n'$ leads to a result similar to Theorem \ref{main2}.  Mimicking its proof, we can state the following proposition that will be used in Section~\ref{special}:
\begin{Prop} \label{prop:approxmod}
Let $\al_0<\infty$ be fixed. 
Assume (\ref{hyplambdalim}), (\ref{hyplambda})
~(\ref{hyplambda0}) and~(\ref{eq:Kraftmainbis}) are satisfied.
Then,  the penalized  estimator  $\hat s_{\tilde  m}$ where
\[\tilde m=\arg\min_{m\in \mathcal{ M}_n'}\left\{\ga_n(\hat s_m)+\widetilde\pen_n(m)\right\}\] 
satisfies the following result:
for any $\al\in~(0,\al_0]$,
\[
MS(\tilde s_{\tilde m},\rho_\al):=:\mathcal{ A}_{_{\Mcol}}^{\al} \cap
\mathcal{L}_V^{\al}.
\]
\end{Prop}

\noindent Remark that $\mathcal{M}_n$,  $\mathcal{L}_V^{\al}$ and $\mathcal{
  A}_{_{\Mcols}}^{\al}$ can be defined in a similar fashion for any
arbitrary dictionary $\Phi$. However, one
can only obtain the inclusion $MS(\hat s_{\hat m},\rho_\al) \subset \mathcal{ A}_{_{\Mcols}}^{\al} \cap
\mathcal{L}_V^{\al}$ in the general case.
\section{Comparisons of model selection estimators}\label{part3}
 The aim of this section is twofold. Firstly, we propose to illustrate our previous maxiset results to different model selection estimators built with wavelet methods by identifying precisely the spaces  $\mathcal{ A}_{_{\Mcols}}^{\al}$ and $\mathcal{L}_V^{\al}$. Secondly,  comparisons
between the performances of these estimators are provided and discussed.

\noindent We briefly recall the construction of periodic wavelets bases of the interval $[0,1]$. 
 Let $\phi$ and $\psi$
be  two compactly supported functions of $\mathbb{L}_2(\mathbb{R})$ and denote
for all $j \in \mathbb{N}
$, all $k\in \mathbb{Z} \mbox{ and all } x\in \mathbb{R}$,
$\phi_{jk}(x)=2^{^{j/2}}\phi(2^{^j}x-k)$ and  $\psi_{jk}(x)=2^{^{j/2}}\psi(2^{^j}x-k)$. 
Those functions can be periodized in such a way that
 \[\Psi=\{\phi_{00}, \psi_{jk}:\quad  j\geq 0, \ k \in \{0,\dots,2^j-1\}\}\] constitutes an
orthonormal basis of $\L_{2}([0,1])$.  Some popular examples of such
bases
are given in \cite{daub}.
The function $\phi$ is called the scaling function and $\psi$ the
corresponding wavelet. Any periodic function $s \in \mathbb{L}_2([0,1])$  can be represented as:
\[
s= \alpha_{00} \phi_{00} + \sum_{j=0}^{\infty}\sum_{k=0}^{2^j-1}\beta_{jk}\psi_{jk}
\]
 where
\[\alpha_{00}=\begin{displaystyle}\int\end{displaystyle}_{[0,1]}s(t)\phi_{00}(t)dt\]
and for any 
$j   \in    \mathbb{N}$   and   for   any   $k    \in   \{0,\dots,2^j-1\}$
\[\beta_{jk}=\begin{displaystyle}\int\end{displaystyle}_{[0,1]}s(t)\psi_{jk}(t)dt.\]
Finally, we  recall the  characterization of Besov  spaces using wavelets.
Such spaces will  play an important role in the following. In this section
we assume that 
the multiresolution analysis associated with
the basis $\Psi$ is $r$-regular with $r\geq 1$ as defined in~\cite{mey}.  In this case, for any $0<\al<r$
and any $1\leq p,q\leq\infty$, the periodic function $s$ belongs to the Besov space
$\mathcal{B}^\al_{p,q}$ if and only if 
$|\alpha_{00}|<\infty$ and
\[\sum_{j=0}^\infty 2^{jq(\al+\frac{1}{2}-\frac{1}{p})}\|\beta_{j.}\|_{\ell_p}^q<\infty\quad
\mbox{if } q<\infty,
\]
\[\sup_{j\in \mathbb{N}}2^{j(\al+\frac{1}{2}-\frac{1}{p})}\|\beta_{j.}\|_{\ell_p}<\infty\quad
\mbox{if } q=\infty\]
where $(\beta_{j.})=(\beta_{jk})_k$. This characterization allows to recall the following embeddings:
\[\mathcal{B}^\al_{p,q}\subsetneq \mathcal{B}^{\al'}_{p',q'} \mbox{ as
  soon as } \al-\frac{1}{p}\geq \alpha'-\frac{1}{p'}, \ p < p' \mbox{ and }  q\leq q'\]
and
\[\mathcal{B}^\al_{p,\infty}\subsetneq\mathcal{B}^{\al}_{2,\infty}\mbox{
  as soon as }p > 2.\]
\subsection{Collection of Sieves}\label{nested}
We consider first a single model collection corresponding to a class of nested models
\begin{equation*}\label{colnested}
\Msievecol = \{m=\mbox{span}\{\phi_{00}, \psi_{jk}: \: \: j < N_m, 0\leq k <2^j\}: \: \: N_m \in \mathbb{N}\}.
\end{equation*}
For  such a
model collection,  Theorem~\ref{main2}
could be applied with $V_n=\L_2$. One can even remove
Assumption~\eqref{hyplambda0} which imposes a minimum value on
$\lambda_{n_0}$ that depends on the rate $\rho_\alpha$:
\begin{Prop}\label{nestedtheo}
Let $0<\alpha<r$ and let $\hat{s}^{(s)}_{\hat{m}}$ be the model
selection estimator associated with the model collection
$\Msievecol$. Then, under
Assumptions~\eqref{hyplambdalim},~\eqref{hyplambda} and~\eqref{eq:Kraftmainbis},
\[
MS(\hat{s}^{(s)}_{\hat{m}},\rho_\al):=:\mathcal{B}^{\alpha}_{2,\infty}.
\]
\end{Prop}
\noindent Remark that it suffices to choose $\lambda_n\geq\lambda_0$ with
$\lambda_0$, independent of $\alpha$, large enough to ensure Condition~\eqref{eq:Kraftmainbis}.

\noindent It is important to notice that the estimator $\hat{s}^{(s)}_{\hat{m}}$ cannot be computed in practice 
because to determine the best model $\hat{m}$ one needs to consider an
infinite number of models, which cannot be done without computing an
infinite number of wavelet coefficients. To overcome this issue, we specify a 
maximum
resolution level $j_0(n)$ for estimation where $n\mapsto j_0(n)$ is
non-decreasing. This modification is also in
the scope of
Theorem~\ref{main2}: it corresponds to
\[V_n=\mbox{span}\{\phi_{00}, \psi_{jk}: \: \: 0 \leq j < j_0(n), \:
0\leq k <2^j\}\] and the  model collection $\Msievecol_n$ defined as follows:
\begin{eqnarray*}
\Msievecol_n&=&\Msievecolprime_n=\{m \in \Msievecol: \: N_m<j_0(n)\}.
\end{eqnarray*}

\noindent  For the specific choice 
\begin{equation}\label{j0trunc} 2^{j_0(n)} \leq n\lambda_n^{-1}< 2^{j_0(n)+1},
\end{equation}
\noindent we obtain:
\begin{align*} 
\mathcal{L}^\alpha_{V}&=\{s=\alpha_{00} \phi_{00} + \sum_{j= 0}^\infty\sum_{k=0}^{2^j-1}\beta_{jk}\psi_{jk}\in\L_2: \: \sup_{n \in \mathbb{N}^*} \: 2^{\frac{2 j_0(n) \alpha}{1+2\alpha}}\|s-P_{V_n}s\|^2<\infty\}\\
&=\{s=\alpha_{00} \phi_{00} + \sum_{j= 0}^\infty\sum_{k=0}^{2^j-1}\beta_{jk}\psi_{jk}\in\L_2: \: \sup_{n \in \mathbb{N}^*} \: 2^{\frac{2 j_0(n) \alpha}{1+2\alpha}}\sum_{j\geq j_0(n)}\sum_k\beta_{jk}^2<\infty\}\\
&=\mathcal{B}^{\frac{\alpha}{1+2\alpha}}_{2,\infty}.
\end{align*}
Since $\mathcal{B}^{\frac{\alpha}{1+2\alpha}}_{2,\infty}
\cap \mathcal{B}^{\alpha}_{2,\infty}$ reduces to
$\mathcal{B}^{\alpha}_{2,\infty}$, arguments of the proofs of Theorem~\ref{main2} and Proposition~\ref{nestedtheo} give:
\begin{Prop}\label{larms0}
Let $0<\alpha<r$ and let $\hat{s}^{(st)}_{\hat{m}}$ be the model selection estimator associated with the model collection $\Msievecol_n$. Then, under Assumptions \eqref{hyplambdalim}, \eqref{hyplambda} and \eqref{eq:Kraftmainbis}
\[MS(\hat{s}^{(st)}_{\hat{m}},\rho_\al):=:\mathcal{B}^{\alpha}_{2,\infty}.\]
\end{Prop}

\noindent This tractable procedure is thus as efficient as the original one. 
We  obtain   the  maxiset  behavior   of  the  non  adaptive   linear  wavelet
procedure pointed out in \cite{riv-a} but here the procedure is completely data-driven.
\subsection{The largest model collections}\label{strut}
In this paragraph we enlarge the model collections in order
to obtain much larger maxisets.
We start with the following
model collection
\[
\Mwealthcol = \{m=\mbox{span}\{\phi_{00}, \psi_{jk}: \: \: (j,k) \in
\mathcal{I}_m\}: \: \mathcal{I}_m \in \mathcal{P}(\mathcal{I}) \}
\]
where 
\[
\mathcal{I}=\bigcup_{j\geq 0} \{ (j,k):\quad k \in \{0,1,\dots,2^j-1\}\}
\]
 and
$\mathcal{P}(\mathcal{I})$ is the set of all subsets of $\mathcal{I}$.
This model collection is so rich that whatever the
sequence $(\lambda_n)_n$, Condition~\eqref{eq:Kraftmainbis} (or even
Condition~\eqref{kraft}) is not satisfied.
To reduce the cardinality of the collection, we restrict the maximum
resolution level to the resolution level $j_0(n)$ defined in (\ref{j0trunc}) and consider
the collections $\Mwealthcol_n$ defined from
$\Mwealthcol$  by
\begin{eqnarray*}
\Mwealthcol_n&=&\Mwealthcolprime_n=\left\{m \in \Mwealthcol:\quad \mathcal{I}_m \in \mathcal{P}(\mathcal{I}^{j_0})
\right\}
\end{eqnarray*}
where
\[
 \mathcal{I}^{j_0} =
    \bigcup_{0 \leq j < j_0(n)} \{ (j,k): \quad k \in \{0,1,\dots,2^j-1\}\}.
\]

\noindent Remark that this corresponds to the same choice of $V_n$ as in the
previous paragraph and the corresponding estimator fits perfectly within the framework of Theorem~\ref{main2}.

\noindent 
The classical logarithmic penalty 
\begin{eqnarray*}\label{stpen}
\pen_n(m)=\frac{\lambda_0 \log(n) D_m}{n},
\end{eqnarray*}
which corresponds to $\lambda_n=\lambda_0 \log(n)$,
 is sufficient to ensure
Condition (\ref{eq:Kraftmainbis}) as soon as $\lambda_0$ is a constant large enough 
(the choice $\lambda_n=\lambda_0$ is not sufficient). 
The identification of  the
corresponding maxiset 
focuses on the characterization of
the space $\mathcal{A}_{{\Mwealthcol}}^{\al}$  since,
as previously,
$\mathcal{L}_V^{\al}=\mathcal{B}^{\frac{\alpha}{1+2\alpha}}_{2,\infty}$.
We rely on  sparsity properties of
$\mathcal{A}_{{\Mwealthcol}}^{\al}$.
 In our  context, sparsity means that  there is a
 \emph{small} proportion  of \emph{large} coefficients of a
signal. Let introduce for, for $n\in\Ne^*$, the notation
\[|\beta|_{(n)}=\inf\left\{u: \quad \mbox{card}\left\{(j,k)\in \mathbb{N}\times\{0,1,\dots,2^j-1\}: \ |\beta_{jk}|>u\right\}<n\right\}\]
to represent the non-increasing rearrangement of the wavelet coefficient of a periodic signal $s$:
\[|\beta|_{(1)}\geq |\beta|_{(2)}\geq\cdots\geq|\beta|_{(n)}\geq\cdots.\]
As the best model $m\in\Mwealthcol$ of prescribed dimension $M$ is
obtained by choosing the subset of index corresponding to the $M$ largest
wavelet coefficients,  
a simple identification of the space $\mathcal{A}_{{\Mwealthcol}}^{\al}$ is
\[\mathcal{
  A}_{{\Mwealthcol}}^{\al}=
  \left\{s=\alpha_{00} \phi_{00} + \sum_{j= 0}^\infty\sum_{k=0}^{2^j-1}\beta_{jk}\psi_{jk}\in\L_2:\quad
  \sup_{M\in\Ne^*}\: M^{2\alpha}\sum_{i=M+1}^{\infty}|\beta|_{(i)}^2 <\infty\right\}.\]
Theorem 2.1 of \cite{kp-a} provides a characterization of this space as a
weak Besov space:
\[\mathcal{  A}_{{\Mwealthcol}}^{\al}=\mathcal{W}_{\frac{2}{1+2\al}}\]
with for any  $q \in ]0,2[$,
\begin{eqnarray*}\label{WB1}
\mathcal{W}_q=\left\{s=\alpha_{00} \phi_{00} + \sum_{j= 0}^\infty\sum_{k=0}^{2^j-1}\beta_{jk}\psi_{jk}\in\L_2:\quad
\sup_{n\in\Ne^*}n^{1/q}|\beta|_{(n)}<\infty\right\}. 
\end{eqnarray*}
Following their definitions, the larger  $\al$, the smaller $q=2/(1+2\al)$ and  the sparser  the sequence  $(\beta_{jk})_{j,k}$. 
Lemma 2.2 of \cite{kp-a} shows that the spaces $\mathcal{W}_q$ ($0<q<2$) have other
characterizations in terms of wavelet coefficients:
\begin{eqnarray*}\label{WB2}
 \mathcal{W}_q&=&\left\{s=\alpha_{00} \phi_{00} + \sum_{j= 0}^\infty\sum_{k=0}^{2^j-1}\beta_{jk}\psi_{jk}\in\L_2:\quad
 \sup_{u>0}u^{q-2}\sum_{j}\sum_k\beta_{jk}^21_{|\beta_{jk}|\leq u}<\infty\right\}\\
&=&\left\{s=\alpha_{00} \phi_{00} + \sum_{j= 0}^\infty\sum_{k=0}^{2^j-1}\beta_{jk}\psi_{jk}\in\L_2:\quad
 \sup_{u>0}u^{q}\sum_{j}\sum_k1_{|\beta_{jk}|> u}<\infty\right\}\nonumber.
\end{eqnarray*}
We obtain thus the following proposition.
\begin{Prop}\label{larms1}
Let $\alpha_0<r$ be fixed, let $0<\alpha\leq \alpha_0$ and let $\hat{s}^{(l)}_{\hat{m}}$ be the model selection estimator associated with the model collection $\Msievecol_n$. Then, under Assumptions \eqref{hyplambdalim}, \eqref{hyplambda}, \eqref{hyplambda0} and \eqref{eq:Kraftmainbis}:
\[
MS\left(\hat{s}^{(l)}_{\hat{m}},\rho_{\alpha}\right):=:\mathcal{B}^{\frac{\al}{1+2\al}}_{2,\infty} \cap \mathcal{W}_\frac{2}{1+2\al}. 
\]
\end{Prop} 
  
\noindent Observe that the estimator $\hat s_{\hat m}^{(l)}$ is easily tractable from a computational
point of view as  the minimization can be rewritten coefficientwise:
\begin{eqnarray*}
\hat{m}(n)&=&
\argmin_{m\in\Mwealthcol_n}  \left\{\gamma_n(\hat s_{m}) + \frac{\lambda_n}{n} D_m\right\}\\
&=& \argmin_{m\in\Mwealthcol_n} \left\{\sum_{j=0}^{j_0(n)-1}\sum_{k=0}^{2^j-1} \left( \hat \beta_{jk}^2 \mathbf{1}_{(j,k)\notin   \mathcal{I}_m}  +
  \frac{\lambda_n}{n}
\mathbf{1}_{(j,k) \in \mathcal{I}_m}\right)\right\}.\end{eqnarray*}
The best subset $\mathcal{I}_{\hat m}$ is thus the set $\{ (j,k) \in
  \mathcal{I}^{j_0}:\quad |\hat\beta_{jk}|>\sqrt{\lambda_n/n}\}$
  and $\hat s^{(l)}_{\hat m}$ corresponds to the well-known hard thresholding estimator,
\[\hat{s}^{(l)}_{\hat{m}}=\hat{\alpha}_{00}\phi_{00}+\sum_{j=0}^{j_0(n)-1}\sum_{k=0}^{2^j-1}\hat{\beta}_{jk}\mathbf{1}_{_{|\hat{\beta}_{jk}|>\sqrt{\frac{\lambda_n}{n}}}} \ \psi_{jk}.
\]
  Proposition \ref{larms1} corresponds thus to the maxiset result
  established by Kerkyacharian and Picard\cite{kp-a}.
\subsection{A special strategy for Besov spaces}\label{special}
We consider now the model collection proposed by Massart~\cite{mas}. This collection can be viewed as an hybrid collection between the collections of Sections~\ref{nested} and~\ref{strut}. This strategy turns 
out to be minimax for all Besov spaces
$\mathcal{B}^{\alpha}_{p,\infty}$ when $\alpha > \max(1/p -
1/2,0)$ and $1\leq p \leq \infty$.

\noindent More precisely, for a chosen $\theta>2$, define the model collection by 
\begin{eqnarray*}\label{lastmodel}
\Mstratcol = \{ m=\mbox{span}\{\phi_{00}, \psi_{jk}: \: \: (j,k) \in
\mathcal{I}_m\}: J \in \mathbb{N}, \: \mathcal{I}_m\in
\mathcal{P}_J(\mathcal{I})\},\end{eqnarray*}
where for any $J \in \mathbb{N},$ $\mathcal{P}_J(\mathcal{I})$ is the
set of all subsets $\mathcal{I}_m$ of $\mathcal{I}$ that can be
written
\begin{align*}
\mathcal{I}_m=& \left\{(j,k):\quad 0 \leq j < J, 0\leq k<2^j \right\}\\
&\quad\bigcup \cup_{j\geq J}
\left\{ (j,k):\quad k \in A_{j}, |A_{j}|=\lfloor2^{J}(j-J+1)^{-\theta}\rfloor
\right\}
\end{align*}
with $\lfloor x \rfloor := \max\{ n \in \mathbb{N}:\ n \leq x\}$.

\noindent As remarked in \cite{mas}, for any $J \in \mathbb{N}$ and
any $\mathcal{I}_m \in \mathcal{P}_J(\mathcal{I})$,
the dimension $D_{m}$ of the corresponding model $m$ depends only on $J$ and is  such that
\[2^J \leq D_{m} \leq 2^J \left(1+\sum_{n\geq 1}n^{-\theta}
\right).\]
We denote by $D_J$ this common dimension.
Note that the model collection $\Mstratcol$ does not vary with
$n$. Using Theorem~\ref{main2} with $V_n=\L_2$, we
have the following proposition. 
\begin{Prop}\label{massarttheo}
Let $\alpha_0<r$ be fixed, let $0<\alpha\leq \alpha_0$ and let $\hat{s}^{(h)}_{\hat{m}}$ be the model selection estimator associated with the model collection $\Mstratcol$. Then, under Assumptions \eqref{hyplambdalim}, \eqref{hyplambda}, \eqref{hyplambda0} and \eqref{eq:Kraftmainbis}:
\[
MS\left(\hat{s}^{(h)}_{\hat{m}},\rho_{\alpha}\right):=:\mathcal{
  A}_{_{\Mstratcol}}^{\al}, 
\]
with
\begin{align*}
\mathcal{
A}_{_{\Mstratcol}}^{\al} &=
\left\{s=\alpha_{00} \phi_{00} +
\sum_{j\geq
  0}\sum_{k=0}^{2^j-1}\beta_{jk}\psi_{jk}\in\L_2:
\vphantom{\sup_{J \geq 0}\left\{2^{2J\alpha}\sum_{j\geq J} \sum_{k\geq \lfloor2^{J}(j-J+1)^{-\theta}\rfloor} |\beta_{j}|_{(k)}^2\right\}<\infty}
\right.\\
&\quad\quad\quad\quad \left. 
\sup_{J \geq 0}\: 2^{2J\alpha}\sum_{j\geq J} \sum_{k\geq
 \lfloor2^{J}(j-J+1)^{-\theta}\rfloor}
 |\beta_{j}|_{(k)}^2<\infty
\right\},
\end{align*}
where $(|\beta_{j}|_{(k)})_k$ is the reordered sequence of coefficients $(\beta_{jk})_k$: \[|\beta_{j}|_{(1)}\geq|\beta_{j}|_{(2)}\cdots|\beta_{j}|_{(k)}\geq\cdots\geq|\beta_{j}|_{(2^j)}.\]
\end{Prop}
\noindent Remark that, as in Section \ref{nested}, as soon as   $\lambda_n\geq \lambda_0$ with
$\lambda_0$ large enough, Condition~\eqref{eq:Kraftmainbis} holds.

\noindent This large set cannot be characterized in terms of classical
spaces. Nevertheless it is undoubtedly a large functional space, since
as proved in Section~\ref{sec:space-embeddings},  for every
$\alpha>0$ and every $p\geq
1$ satisfying $p>2/(2\alpha+1)$ we get
 \begin{eqnarray}
\mathcal{B}^\alpha_{p,\infty} &\subsetneq& \mathcal{A}^\alpha_{{\Mstratcol}}. \label{nestesdbattue}
\end{eqnarray}

\noindent This new procedure is not computable since one needs an
infinite number of wavelet coefficients to perform it. The problem of
calculability can be solved by introducing, as previously, a
maximum scale $j_0(n)$ as defined in \eqref{j0trunc}. 
We consider the class of collection models  $(\Mstratcol_n)_n$ defined as follows:
\begin{align*}
\Mstratcol_n &= \{ m=\mbox{span}\{\phi_{00}, \psi_{jk}: \: \: (j,k) \in
\mathcal{I}_m, j<j_0(n) \}:\\
&\quad\quad\quad\quad\quad\quad\quad\quad\quad\quad\quad\quad\quad\quad\quad\quad
 J \in \mathbb{N}, \: \mathcal{I}_m\in
\mathcal{P}_J(\mathcal{I})\}.
\end{align*}
This model collection does not satisfy the embedding condition
$\Mstratcol_n\subset\Mstratcol_{n+1}$. Nevertheless, 
we can use Proposition~\ref{prop:approxmod} with 
\[
\widetilde{\pen}_n(m)=\frac{\lambda_n}{n}D_J
\] if $m$
is obtained from an index subset ${\mathcal I}_{m}$ in $\mathcal{P}_J(\mathcal{I})$.
This slight over-penalization leads to the following result.
\begin{Prop}\label{massartcutheo}
Let $\alpha_0<r$ be fixed, let $0<\alpha\leq \alpha_0$ and let $\hat{s}^{(ht)}_{\tilde{m}}$ be the model selection estimator associated with the model collection $\Mstratcol_n$. Then, under Assumptions \eqref{hyplambdalim}, \eqref{hyplambda}, \eqref{hyplambda0} and \eqref{eq:Kraftmainbis}:
\[
MS\left(\hat{s}^{(ht)}_{\tilde{m}},\rho_{\alpha}\right):=:\mathcal{B}^{\frac{\al}{1+2\al}}_{2,\infty} \cap \mathcal{
  A}_{_{\Mstratcol}}^{\al}. 
\]
\end{Prop}
\noindent Modifying Massart's strategy in order to obtain a practical estimator
changes the maxiset performance.
The previous set $\mathcal{
  A}_{_{\Mstratcol}}^{\al}$ is intersected with the strong Besov space
  $\mathcal{B}^{\al/(1+2\al)}_{2,\infty}$. Nevertheless, as it  will be proved in Section \ref{sec:space-embeddings}, the maxiset
  $MS\left(\hat{s}^{(ht)}_{\tilde{m}},\rho_{\alpha}\right)$ is
  still a large functional space. Indeed, for every
$\alpha>0$ and every $p$ satisfying $p\geq \max(1,2\left(\frac{1}{1+2\alpha}+2\alpha \right)^{-1})$
 \begin{eqnarray}
\mathcal{B}^\alpha_{p,\infty} &\subseteq&
\mathcal{B}^{\frac{\al}{1+2\al}}_{2,\infty} \cap  \mathcal{A}^\alpha_{{\Mstratcol}}. \label{nestesdbattuecut}
\end{eqnarray}
 \subsection{Comparisons of model selection estimators}
In this paragraph, we compare the maxiset performances of the different model selection procedures described previously. For a chosen rate of convergence let us recall that the larger the maxiset, the better the estimator. To begin, we propose to focus  on the model selection estimators which are tractable from the computational point of view. Gathering Propositions \ref{larms0}, \ref{larms1} and \ref{massartcutheo} we obtain the following comparison.
\begin{Prop}\label{vs1}
Let $0<\alpha<r.$ 
\begin{itemize}
\item[-] If for every $n,$ $\lambda_n = \lambda_0 \log(n)$ with $\lambda_0$ large enough, then
\begin{eqnarray}\label{incmax}
MS(\hat{s}^{(st)}_{\hat{m}},\rho_\alpha)\subsetneq MS(\hat{s}^{(ht)}_{\tilde{m}},\rho_\alpha)\subsetneq MS(\hat{s}^{(l)}_{\hat{m}},\rho_\alpha).\end{eqnarray}
\item[-]  If for every $n,$ $\lambda_n=\lambda_0$ with $\lambda_0$ large enough, then
\begin{eqnarray}\label{incmax1}
MS(\hat{s}^{(st)}_{\hat{m}},\rho_\alpha) \subsetneq
MS(\hat{s}^{(ht)}_{\tilde{m}},\rho_\alpha).\end{eqnarray}
\end{itemize}
\end{Prop}  
\noindent It means the followings.
\begin{itemize}
\item[-]  If for every $n,$ $\lambda_n = \lambda_0 \log(n)$ with
  $\lambda_0$ large enough, then, according to the maxiset point of
  view, the estimator
  $\hat{s}^{(l)}_{\hat{m}}$ strictly outperforms the estimator
  $\hat{s}^{(ht)}_{\tilde{m}}$ which strictly outperforms the estimator
  $\hat{s}^{(st)}_{\hat{m}}$.
\item[-]
  If for every $n,$ $\lambda_n = \lambda_0$ or $\lambda_n = \lambda_0 \log(n)$  with
  $\lambda_0$ large enough, then, according to the maxiset point of
  view, the estimator
  $\hat{s}^{(ht)}_{\tilde{m}}$ strictly outperforms the estimator
  $\hat{s}^{(st)}_{\hat{m}}$.
\end{itemize}
The corresponding embeddings of functional spaces are proved in Section \ref{sec:space-embeddings}.
The hard thresholding
estimator $\hat{s}^{(l)}_{\hat{m}}$ appears as the best estimator when $\lambda_n$ grows
logarithmically while estimator $\hat{s}^{(ht)}_{\tilde{m}}$ is the
best estimator when $\lambda_n$ is constant. In both cases, those
estimators perform very well since their maxiset contains all
the Besov spaces $\mathcal{B}^{\frac{\alpha}{1+2\alpha}}_{p,\infty}$
with $p\geq \max\left(1,\left(\frac{1}{1+2\alpha}+2\alpha \right)^{-1}\right)$.

\noindent We forget now the calculability issues and consider 
the maxiset of the original procedure proposed by
Massart. Propositions~\ref{larms1}, \ref{massarttheo}
and~\ref{massartcutheo} lead then to the following result.
\begin{Prop}\label{vs2}
Let $0<\alpha<r$. 
\begin{itemize}
\item[-] If  for any $n,$ $\lambda_n = \lambda_0 \log(n)$ with $\lambda_0$ large enough
then
\begin{align}
\label{incmax2}
MS(\hat{s}^{(h)}_{\hat{m}},\rho_\alpha) \not\subset
MS(\hat{s}^{(l)}_{\hat{m}},\rho_\alpha)
\quad\text{and}\quad
MS(\hat{s}^{(l)}_{\hat{m}},\rho_\alpha) \not\subset
MS(\hat{s}^{(h)}_{\hat{m}},\rho_\alpha).
\end{align}
\item[-] If  for any $n,$ $\lambda_n = \lambda_0$ or $\lambda_n =
  \lambda_0 \log(n)$ with $\lambda_0$ large enough
then
\begin{eqnarray}\label{incmax3}
MS(\hat{s}^{(ht)}_{\tilde{m}},\rho_\alpha) \subsetneq MS(\hat{s}^{(h)}_{\hat{m}},\rho_\alpha).\end{eqnarray}
\end{itemize}
\end{Prop}
\noindent Hence, within the maxiset framework,  the estimator $\hat{s}^{(h)}_{\hat{m}}$ strictly outperforms
the estimator $\hat{s}^{(ht)}_{\tilde{m}}$
while
the estimators $\hat{s}^{(h)}_{\hat{m}}$ and $\hat{s}^{(l)}_{\hat{m}}$
are not comparable. Note that we did not consider the maxisets of
the estimator $\hat{s}^{(s)}_{\hat{m}}$ in this section as they are identical
to the ones of the tractable estimator $\hat{s}^{(st)}_{\hat{m}}$.
\begin{figure}[ht]
  \centering
\includegraphics[width=13cm]{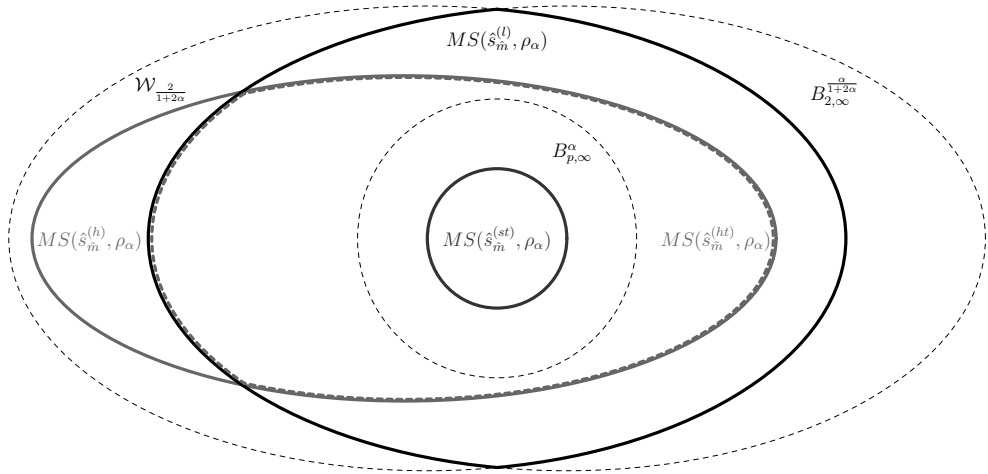}
  \caption{Maxiset embeddings when $\lambda_n = \lambda_0 \log(n)$ and $\max(1,2\left(\frac{1}{1+2\alpha}+2\alpha \right)^{-1})\leq p\leq 2$.}
\label{fig:Maxilog}
\end{figure}
\begin{figure}[ht]
  \centering
\includegraphics[width=13cm]{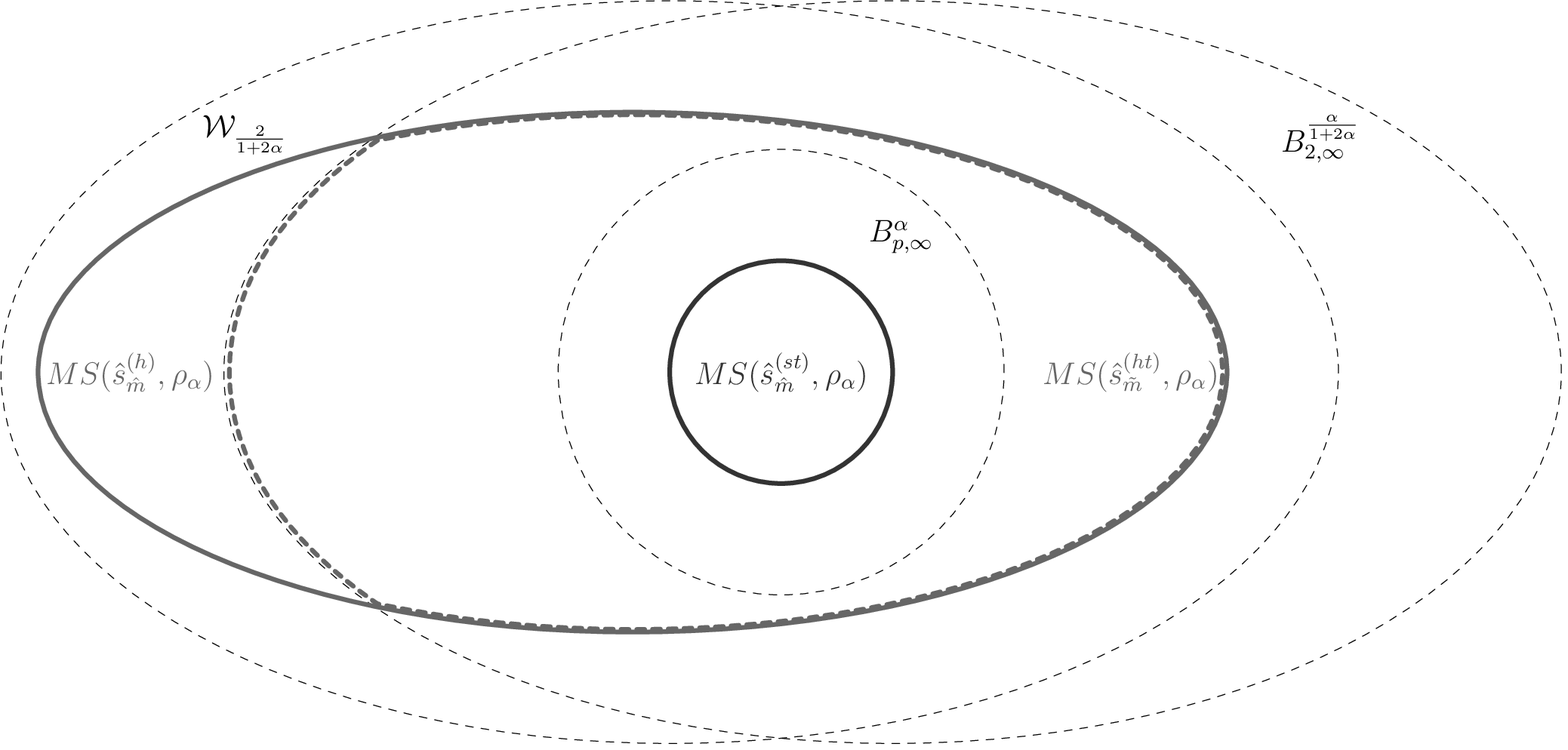}
  \caption{Maxiset embeddings when $\lambda_n = \lambda_0$ and $\max(1,2\left(\frac{1}{1+2\alpha}+2\alpha \right)^{-1})\leq p\leq 2$.}
\label{fig:Maxicst}
\end{figure}
We summarize all those embeddings in
Figure~\ref{fig:Maxilog} and Figure~\ref{fig:Maxicst}:
Figure
\ref{fig:Maxilog} represents these maxiset embeddings for the choice $\lambda_n
= \lambda_0 \log(n)$, while Figure \ref{fig:Maxicst} represents these
maxiset embeddings for the choice $\lambda_n = \lambda_0$.
\section{Proofs}\label{proofs}
For any functions $u$ and $u'$ of $\L_2(\mathcal{D})$, we denote by $\langle u,u'\rangle$ the $\L_2$-scalar product between $u$ and $u'$:
$$\langle u,u'\rangle=\int_\mathcal{D}u(t)u'(t)dt.$$
We denote by $C$ a constant whose value may change at each line.
\subsection{Proof of Theorem \ref{main}}
Without loss of generality, we assume that $n_0=1$. We start by constructing a different representation of the white noise
model. For any model $m$, we define $\Ws_m$,
the projection of the noise on $m$ by
$$
\Ws_m = \sum_{i=1}^{D_m} W_{e^m_i} e^m_i,\quad W_{e^m_i}=\int_\mathcal{D} e^m_i(t)dW_t,$$
where $\{e^m_i\}_{i=1}^{D_m}$ is any orthonormal basis of $m$. 
For any
function $s\in m$, we have :
$$
W_s = \int_\mathcal{D} s(t)dW_t= \sum_{i=1}^{D_m} \langle s, e^m_i \rangle W_{e^m_i} = \langle
\Ws_m , s \rangle.
$$
The key observation is now that with high probability,
$\|\Ws_m\|^2$ can be controlled simultaneously over all models. More
precisely, for any $m,m' \in \Mcol_n$, we define the space $m+m'$ as
the space spanned by the functions of $m$ and $m'$ and control the
norm of $\|\Ws_{m+m'}\|^2$.

\begin{Lemma}\label{cirelson}
Let $n$ be fixed and
\[
A_n = \left\{ \sup_{m \in \Mcol_n}\sup_{m' \in \Mcol_n}\left\{(D_m+D_{m'})^{-1} \|
\Ws_{m+m'}\|^2\right\} \leq  \lambda_n \right\}.
\]
Then, under Assumption (\ref{eq:Kraftmain}), we have
$\P\{A_n\}\geq p$.
\end{Lemma}
\begin{proof}
The Cirelson-Ibragimov-Sudakov inequality (see \cite{mas}, page 10)
implies that for any $t>0$, any $m\in \Mcol_n$ and any $m'\in \Mcol_n$
\[
\P\left\{\|\mathbf{W}_{m+m'}\|\geq \E\left[\|\Ws_{m+m'}\|\right]+t\right\}\leq e^{-\frac{t^2}{2}}.
\]
Since 
\[
\E\left[\|\Ws_{m+m'}\|\right]\leq
\sqrt{\E\left[\|\Ws_{m+m'}\|^2\right]}\leq \sqrt{D_m+D_{m'}},
\] with
$t=\sqrt{\lambda_n (D_m+D_{m'})}-\sqrt{D_m+D_{m'}}$, we obtain
\[
\P\left\{\|\Ws_{m+m'}\|^2\geq  \lambda_n (D_m+D_{m'})
\right\}\leq e^{-\frac{\left(\sqrt{\lambda_n}-1\right)^2 (D_m+D_{m'})}{2}}.
\]
Assumption~\eqref{eq:Kraftmain} implies thus that
\begin{align*}
1-\P\{A_n\} &\leq \sum_{m\in\Mcol_n}  \sum_{m'\in\Mcol_n}
\P\left\{\|\Ws_{m+m'}\|^2 \geq 
  \lambda_n (D_m +D_m')\right\}\\
&
\leq 
\sum_{m\in\Mcol_n} \sum_{m'\in\Mcol_n}
e^{-\frac{\left(\sqrt{\lambda_n}-1\right)^2(D_m+D_{m'})}{2}}\\
&  \leq \left(\sum_{m\in\Mcol_n} 
e^{-\frac{\left(\sqrt{\lambda_n}-1\right)^2 D_m}{2}}
\right)^2 \leq 1-p.
\end{align*}
\end{proof}
\noindent We define $m_0(n)$ (denoted $m_0$ when there is no ambiguity), the model
that minimizes a quantity close to $Q(s,n)$:
\begin{align*}
m_0(n) = \argmin_{m\in\Mcol_n} \left\{\|s_m-s\|^2 + \frac{\lambda_n}{Kn} D_m\right\},
\end{align*}
where $K$ is an absolute constant larger than 1 specified later.
The proof of the theorem begins by a bound on
$\|s_{m_0}-s\|^2 $:
\begin{Lemma}\label{lem:lemmastart}
For any $0<\gamma<1$,
\begin{align}
\label{eq:bound2}
\|s_{m_0}-s\|^2 
&\leq \frac{\tilde K + 4\gamma^{-1}}{\tilde K \P\{A_n\}}
\E\left[\|\hat s_{\hat m}-s\|^2\right]
+
\left(
\frac{
K(2\gamma^{-1}+1)
}{\tilde K \P\{A_n\}}
+ \frac{2K\gamma\lambda_n}{\tilde K}
\right)
 \frac{D_{m_0}}{Kn}
\end{align}
if the constant  $\tilde K = K(1-\gamma)-2\gamma^{-1}-1$ satisfies $\tilde K>0$.
\end{Lemma}
\begin{proof}
By definition,
$$
  \gamma_n(\hat s_{\hat m}) + \lambda_n \frac{D_{\hat m}}{n} \leq \gamma_n(\hat s_{m_0}) + \lambda_n \frac{D_{m_0}}{n}.$$
Thus,
\begin{align*}
 \lambda_n \frac{D_{\hat m} - D_{m_0}}{n} 
 & \leq
 \gamma_n(\hat s_{m_0}) - \gamma_n(\hat s_{\hat m})\\
 & \leq -2 Y_n(\hat s_{m_0}) + \|\hat s_{m_0}\|^2
 + 2 Y_n(\hat s_{\hat m}) - \|\hat s_{\hat m}\|^2 \\
& \leq - 2 \langle \hat s_{m_0} , s \rangle + \|\hat s_{m_0}\|^2
 + 2 \langle \hat s_{\hat m} , s \rangle - \|\hat s_{\hat m}\|^2 +
 \frac{2}{\sqrt{n}} W_{\hat s_{\hat m} - \hat s_{m_0}}\\
& \leq \|\hat s_{m_0}-s\|^2 - \|\hat s_{\hat m}-s\|^2 +
\frac{2}{\sqrt{n}} W_{\hat s_{\hat m} - \hat s_{m_0}}. 
\end{align*}
Let $0<\gamma<1$. As $\hat s_{\hat m} - \hat
  s_{m_0}$ is supported by the space $\hat m + m_0$ spanned by the functions of $\hat
  m$ and $m_0$, we obtain with the previous definition
\begin{align*}
 \lambda_n \frac{D_{\hat m} - D_{m_0}}{n} 
 & 
\leq \|\hat s_{m_0}-s\|^2 - \|\hat s_{\hat m}-s\|^2 +
\frac{2}{\sqrt{n}} \langle \Ws_{\hat m + m_0} , \hat s_{\hat m} - \hat
  s_{m_0} \rangle\\
\begin{split}
& \leq \|\hat s_{m_0}-s\|^2 - \|\hat s_{\hat m}-s\|^2 +
\frac{\gamma}{n} \|\Ws_{\hat m + m_0}\|^2\\ 
&\qquad\qquad+ \frac{2}{\gamma}\left( \|\hat s_{m_0}-s\|^2 + \|\hat
s_{\hat m}-s\|^2 \right)
\end{split}\\
& \leq \left(\frac{2}{\gamma}+1\right) \|\hat s_{m_0}-s\|^2 + \left(\frac{2}{\gamma}-1\right)
\|\hat s_{\hat m}-s\|^2 + \frac{\gamma}{n} \|\Ws_{\hat m + m_0}\|^2.
\end{align*}
We multiply now by  $\mathbf{1}_{A_n}$ to obtain
\begin{align*}
\begin{split}
 \lambda_n \mathbf{1}_{A_n} \frac{D_{\hat m} - D_{m_0}}{n} 
& \leq \left(\frac{2}{\gamma}+1\right) \mathbf{1}_{A_n} \|\hat s_{m_0}-s\|^2 + \left(\frac{2}{\gamma}-1\right)
\mathbf{1}_{A_n} \|\hat s_{\hat m}-s\|^2 \\
&\qquad\qquad+ \mathbf{1}_{A_n}
\frac{\gamma}{n} \|\Ws_{\hat m + m_0}\|^2.
\end{split}
\\
\intertext{Using now the definition of $A_n$ and Lemma \ref{cirelson}, it yields}
\begin{split}
 \lambda_n \mathbf{1}_{A_n} \frac{D_{\hat m} - D_{m_0}}{n} 
& \leq \left(\frac{2}{\gamma}+1\right) \mathbf{1}_{A_n} \|\hat s_{m_0}-s\|^2 + \left(\frac{2}{\gamma}-1\right)
\mathbf{1}_{A_n} \|\hat s_{\hat m}-s\|^2 \\
&\qquad\qquad+
 \gamma \lambda_n  \mathbf{1}_{A_n} \frac{D_{\hat m} +
   D_{m_0}}{n}
\end{split}
\end{align*}
and thus
\begin{align*}
\begin{split}
\left(1-\gamma\right) \lambda_n \mathbf{1}_{A_n} \frac{D_{\hat m} - D_{m_0}}{n} 
& \leq   \left(\frac{2}{\gamma}+1\right) \mathbf{1}_{A_n} \|\hat s_{m_0}-s\|^2
\\
&\qquad+ \left(\frac{2}{\gamma}-1\right)
\mathbf{1}_{A_n} \|\hat s_{\hat m}-s\|^2 +
 2\gamma\lambda_n  \mathbf{1}_{A_n} \frac{D_{m_0}}{n}.
\end{split}
\end{align*}
One obtains
\begin{align}
\begin{split}
\label{eq:bound1}
\lambda_n \mathbf{1}_{A_n} \frac{D_{\hat m} - D_{m_0}}{n} 
& \leq   \frac{\frac{2}{\gamma}+1}{1-\gamma} \mathbf{1}_{A_n} \|\hat s_{m_0}-s\|^2 + \frac{\frac{2}{\gamma}-1}{1-\gamma}
\mathbf{1}_{A_n} \|\hat s_{\hat m}-s\|^2 \\
&\qquad\qquad+
 \frac{2\gamma}{1-\gamma} \lambda_n  \mathbf{1}_{A_n}
 \frac{D_{m_0}}{n}.
\end{split}
\end{align}
We derive now a bound on $\|s_{m_0}-s\|^2$. By definition,
  \begin{align*}
  \|s_{m_0}-s\|^2 + \lambda_n \frac{D_{m_0}}{Kn} \leq \|s_{\hat m}-s\|^2 +
  \lambda_n \frac{D_{\hat m}}{Kn}
\end{align*}
and thus
\begin{align*}
  \|s_{m_0}-s\|^2 &\leq \|s_{\hat m}-s\|^2 +
  \lambda_n \frac{D_{\hat m} - D_{m_0}}{Kn}.
\end{align*}
By multiplying by $\mathbf{1}_{A_n}$  and plugging the bound~(\ref{eq:bound1}), we have:
\begin{align*}
\mathbf{1}_{A_n}  \|s_{m_0}-s\|^2 &\leq \mathbf{1}_{A_n} \|s_{\hat m}-s\|^2 +
  \lambda_n \mathbf{1}_{A_n}\frac{D_{\hat m} - D_{m_0}}{Kn}\\
&\leq\mathbf{1}_{A_n} \|s_{\hat m}-s\|^2 +
 \frac{\frac{2}{\gamma}+1}{K(1-\gamma)}
 \mathbf{1}_{A_n} \|\hat s_{m_0}-s\|^2 \\
&\qquad+ \frac{\frac{2}{\gamma}-1}{K(1-\gamma)}
\mathbf{1}_{A_n} \|\hat s_{\hat m}-s\|^2 
+
 \frac{2\gamma}{K(1-\gamma)} \lambda_n  \mathbf{1}_{A_n} 
 \frac{D_{m_0}}{n}\\
&\leq \left( 1 + \frac{\frac{2}{\gamma}-1}{K(1-\gamma)}\right)
\mathbf{1}_{A_n} \|\hat s_{\hat m}-s\|^2\\
&\qquad+ \frac{\frac{2}{\gamma}+1}{K(1-\gamma)}
\mathbf{1}_{A_n} ( \|s_{m_0}-s\|^2 + \frac{1}{n}\|W_{m_0}\|^2)\\
&\qquad\qquad
+
 \frac{2\gamma}{K(1-\gamma)} \lambda_n  \mathbf{1}_{A_n} 
 \frac{D_{m_0}}{n}\\
& \leq \left( 1 + \frac{\frac{2}{\gamma}-1}{K(1-\gamma)}\right)
\|\hat s_{\hat m}-s\|^2 
+ \frac{\frac{2}{\gamma}+1}{K(1-\gamma)}
\mathbf{1}_{A_n} \|s_{m_0}-s\|^2 \\
&\qquad\qquad+ \frac{\frac{2}{\gamma}+1}{K(1-\gamma)}
\frac{1}{n}\|W_{m_0}\|^2 +
 \frac{2\gamma}{K(1-\gamma)} \lambda_n  \mathbf{1}_{A_n} 
 \frac{D_{m_0}}{n}
\end{align*}
and thus
\begin{align*}
\begin{split}
\left(1 -
  \frac{\frac{2}{\gamma}+1}{K(1-\gamma)}\right)\mathbf{1}_{A_n}
&\|s_{m_0}-s\|^2 \\
&\leq \left( 1 + \frac{\frac{2}{\gamma}-1}{K(1-\gamma)}\right)
\|\hat s_{\hat m}-s\|^2\\
&\qquad
+ \frac{\frac{2}{\gamma}+1}{K(1-\gamma)}
\frac{1}{n}\|W_{m_0}\|^2
+
 \frac{2\gamma}{K(1-\gamma)} \lambda_n  \mathbf{1}_{A_n} 
 \frac{D_{m_0}}{n}.
\end{split}
\end{align*}
Taking the expectation on both sides yields
\begin{align*}
\begin{split}
\left(1 -
  \frac{\frac{2}{\gamma}+1}{K(1-\gamma)}\right)&
\P\{A_n\}
\|s_{m_0}-s\|^2 \\
 & \leq \left( 1 + \frac{\frac{2}{\gamma}-1}{K(1-\gamma)}\right)
 \E\left[\|\hat s_{\hat m}-s\|^2\right]
\\
&\qquad
+ \left(
\frac{\frac{2}{\gamma}+1}{K(1-\gamma)}
+
 \frac{2\gamma}{K(1-\gamma)} \P\{A_n\} \lambda_n  
\right)
   \frac{D_{m_0}}{n}
\end{split}
\end{align*}
and thus as soon as $1 -
  \frac{\frac{2}{\gamma}+1}{K(1-\gamma)}>0$
\begin{align*}
\begin{split}
\|s_{m_0}-s\|^2 
&\leq
\frac{1 + \frac{\frac{2}{\gamma}-1}{K(1-\gamma)}}
{\left(1 -
  \frac{\frac{2}{\gamma}+1}{K(1-\gamma)}\right)\P\{A_n\}}  
\E\left[\|\hat s_{\hat m}-s\|^2\right]\\
&\qquad\qquad
+ \frac{\frac{\frac{2}{\gamma}+1}{K(1-\gamma)}
+
 \frac{2\gamma}{K(1-\gamma)} \P\{A_n\} \lambda_n  
}{\left(1 -
  \frac{\frac{2}{\gamma}+1}{K(1-\gamma)}\right)\P\{A_n\}} 
 \frac{D_{m_0}}{n}
\end{split}
\\
\begin{split}
& \leq
\frac{K(1-\gamma)+\frac{2}{\gamma}-1}{\left(K(1-\gamma)-\frac{2}{\gamma}-1\right)\P\{A_n\}}
\E\left[\|\hat s_{\hat m}-s\|^2\right]\\
&\qquad\qquad
+
\frac{
\frac{2}{\gamma}+1+2\gamma \P\{A_n\}\lambda_n
}{\left(K(1-\gamma)-\frac{2}{\gamma}-1\right)\P\{A_n\}}
 \frac{D_{m_0}}{n}
\end{split}\\
&\leq \frac{\tilde K + \frac{4}{\gamma}}{\tilde K \P\{A_n\}}
\E\left[\|\hat s_{\hat m}-s\|^2\right]
+
\frac{
\frac{2}{\gamma}+1+2\gamma \P\{A_n\}\lambda_n
}{\tilde K \P\{A_n\}}
 \frac{D_{m_0}}{n}\\
\intertext{which yields}
\|s_{m_0}-s\|^2 
&\leq \frac{\tilde K + \frac{4}{\gamma}}{\tilde K \P\{A_n\}}
\E\left[\|\hat s_{\hat m}-s\|^2\right]
+
\left(
\frac{
K(\frac{2}{\gamma}+1)
}{\tilde K \P\{A_n\}}
+ \frac{2K\gamma\lambda_n}{\tilde K}
\right)
 \frac{D_{m_0}}{Kn}
\end{align*}
with $\tilde K = K(1-\gamma)-\frac{2}{\gamma}-1$.
\end{proof}
\noindent Now, let us specify the constants. We take
\[g(\de,\al_0)=\inf_{\al\in (0,\al_0]}\inf_{x\in \left[\frac{1}{2},1-\delta\right]}\left\{x^{\frac{2\al}{2\al+1}}-x\right\}=(1-\de)^{\frac{2\al_0}{2\al_0+1}}-1+\de\in (0,1).\]
Then we take 
\[\ga=\frac{1}{8}g(\de,\al_0) \mbox{ and } K=\frac{\frac{2}{\ga}+1}{\frac{1}{2}-\ga}.\]
This implies $\tilde K=\frac{K}{2}$ and assumptions of the previous lemma are satisfied. We consider now the dependency of $m_0$ on $n$
and prove by induction the following lemma.
\begin{Lemma}\label{induction} If there exists $C_1 > 0$  such that for any $n$,
\begin{align*}
 \E\left[ \|\hat s_{\hat m(n/2)}-s\|^2\right] &\leq C_1
 \left(\frac{2\lambda_{n/2}}{n}\right)^{\frac{2\alpha}{2\alpha+1}}
\end{align*}
then, provided $\lambda_1\geq \Upsilon(\de,p,\al_0)$, where
\begin{equation}\label{erwan}
\Upsilon(\de,p,\al_0)=\frac{8}{p g(\de,\al_0)}\left(\frac{16}{ g(\de,\al_0)}+1\right),
\end{equation}
there exists a constant $C_2$   such that for any $n$,
  \begin{align*}
\|s_{m_0(n)}-s\|^2 +  \lambda_{n} 
\frac{D_{m_0(n)}}{Kn}
&\leq C_2
 \left(\frac{\lambda_{n}}{n}\right)^{\frac{2\alpha}{2\alpha+1}}.
  \end{align*}
\end{Lemma}
\begin{proof}
By using $\Mcol_{n/2}\subset \Mcol_n$ and (\ref{eq:bound2}), for any $\beta\in[0,1]$, 
if we denote
$$A=\|s_{m_0(n)}-s\|^2 + \lambda_n \frac{D_{m_0(n)}}{Kn},$$ we have
\begin{align*}
A
& \leq   \|s_{m_0(n/2)}-s\|^2 + \lambda_n \frac{D_{m_0(n/2)}}{Kn}\\
\begin{split}
&\leq \beta \|s_{m_0(n/2)}-s\|^2 + (1-\beta)  \|s_{m_0(n/2)}-s\|^2+ \frac{\lambda_n}{2\lambda_{n/2}} \lambda_{n/2}
\frac{2D_{m_0(n/2)}}{Kn}
\end{split}\\
\begin{split}
 &\leq \beta \frac{\tilde K + \frac{4}{\gamma}}{\tilde K \P\{A_{n/2}\}}
 \E\left[ \|\hat s_{\hat m(n/2)}-s\|^2\right]
 + (1-\beta) \|s_{m_0(n/2)}-s\|^2
\\
&\qquad
+ \left( \beta 
\left(\frac{
K(\frac{2}{\gamma}+1)
}{\tilde K \P\{A_{n/2}\}\lambda_{n/2}}
+ \frac{2K\gamma}{\tilde K}
\right)
+ \frac{\lambda_n}{2\lambda_{n/2}}
\right)
  \frac{2 \lambda_{n/2} D_{m_0(n/2)}}{Kn}.
\end{split}
\end{align*}
As $\lambda_n \leq 2 \lambda_{n/2}$, there exists $\beta_n\in[0,1]$ such that
\begin{align*}
1-\beta_n = \beta_n 
\left(\frac{
K(\frac{2}{\gamma}+1)
}{\tilde K \P\{A_{n/2}\}\lambda_{n/2}}
+ \frac{2K\gamma}{\tilde K}
\right)
+ \frac{\lambda_n}{2\lambda_{n/2}}
\end{align*}
so that
\begin{align*}
\begin{split}
 A  
& \leq \beta_n \frac{\tilde K + \frac{4}{\gamma}}{\tilde K \P\{A_{n/2}\}}
 \E\left[ \|\hat s_{\hat m(n/2)}-s\|^2\right]\\
&\qquad\qquad
 + (1-\beta_n) \left( \|s_{m_0(n/2)}-s\|^2 +   
\frac{2\lambda_{n/2}D_{m_0(n/2)}}{Kn}\right).
\end{split}
\end{align*}
The induction can now be started. We assume now that for all $n'\leq n-1$
\[\|s_{m_0(n')}-s\|^2 +  \lambda_{n'} 
\frac{D_{m_0(n')}}{Kn'}\leq C_2
 \left(\frac{\lambda_{n'}}{n'}\right)^{\frac{2\alpha}{2\alpha+1}}.\]
By assumption,
\[
 \E\left[ \|\hat s_{\hat m(n/2)}-s\|^2\right] \leq C_1
 \left(\frac{2\lambda_{n/2}}{n}\right)^{\frac{2\alpha}{2\alpha+1}},\]
so that,
\begin{align*}
\begin{split}
 A   
& \leq \beta_n \frac{\tilde K + \frac{4}{\gamma}}{\tilde K \P\{A_{n/2}\}}
C_1
 \left(\frac{2\lambda_{n/2}}{n}\right)^{\frac{2\alpha}{2\alpha+1}} + (1-\beta_n) C_2
 \left(\frac{2\lambda_{n/2}}{n}\right)^{\frac{2\alpha}{2\alpha+1}}
\end{split}
\\
\begin{split}
& \leq 
 \left(
 \beta_n \frac{\tilde K + \frac{4}{\gamma}}{\tilde K \P\{A_{n/2}\}}
 \frac{C_1}{C_2} + 1 - \beta_n
 \right)
 \left(
 \frac{2\lambda_{n/2}}{\lambda_n}
 \right)^{\frac{2\alpha}{2\alpha+1}}
 C_2 \left( \frac{\lambda_n}{n} \right)^{\frac{2\alpha}{2\alpha+1}}.
\end{split}
\end{align*}
So, we have to prove that
\begin{eqnarray*}
 \left(
 \beta_n \frac{\tilde K + \frac{4}{\gamma}}{\tilde K \P\{A_{n/2}\}}
 \frac{C_1}{C_2} + 1 - \beta_n
 \right)
 \left(
 \frac{2\lambda_{n/2}}{\lambda_n}
 \right)^{\frac{2\alpha}{2\alpha+1}} &\leq& 1
\end{eqnarray*}  
or equivalently,
\begin{align*}
\left(
 \beta_n \left(
\frac{\tilde K + \frac{4}{\gamma}}{\tilde K \P\{A_{n/2}\}}
 \frac{C_1}{C_2}
+ 
\frac{
K(\frac{2}{\gamma}+1)
}{\tilde K \P\{A_{n/2}\}\lambda_{n/2}}
+ \frac{2K\gamma}{\tilde K}
\right)
 + \frac{\lambda_n}{2\lambda_{n/2}}
 \right)
\left(
  \frac{2\lambda_{n/2}}{\lambda_n}
  \right)^{\frac{2\alpha}{2\alpha+1}} &\leq 1.
\end{align*}
This condition can be rewritten as
\begin{align*}
 \beta_n \left(
\frac{\tilde K + \frac{4}{\gamma}}{\tilde K \P\{A_{n/2}\}}
 \frac{C_1}{C_2}
+
\frac{
K(\frac{2}{\gamma}+1)
}{\tilde K \P\{A_{n/2}\}\lambda_{n/2}}
+ \frac{2K\gamma}{\tilde K}
\right) 
 \left(
   \frac{2\lambda_{n/2}}{\lambda_n}
   \right)^{\frac{2\alpha}{2\alpha+1}} 
&\leq 1 - \left( \frac{\lambda_n}{2\lambda_{n/2}}\right)^{\frac{1}{2\alpha+1} }
\end{align*}
or 
\begin{align*}
  \lambda_{n/2} \geq 
\frac{K(\frac{2}{\gamma}+1)}{\tilde{K}\P\{A_{n/2}\}}
\left[\frac{1}{\beta_n}
\left(
 \left(\frac{\lambda_n}{2\lambda_{n/2}}\right)^{\frac{2\alpha}{2\alpha+1}}
- \frac{\lambda_n}{2\lambda_{n/2}}
\right)
- 
\frac{\tilde K + \frac{4}{\gamma}}{\tilde K \P\{A_{n/2}\}}
 \frac{C_1}{C_2}
-
\frac{2K\gamma}{\tilde K}
\right]^{-1}
\end{align*}
provided the right member is positive.
Under the very mild assumption $2(1-\delta)\lambda_{n/2} \geq
\lambda_n \geq \lambda_{n/2}$, it is sufficient to ensure that (\ref{erwan}) is true. Indeed,  $\lambda_{n/2}\geq \lambda_1$ and using values of the constants we have
\begin{eqnarray*}
&&\frac{K(\frac{2}{\gamma}+1)}{\tilde{K}\P\{A_{n/2}\}}
\left[\frac{1}{\beta_n}
\left(
 \left(\frac{\lambda_n}{2\lambda_{n/2}}\right)^{\frac{2\alpha}{2\alpha+1}}
- \frac{\lambda_n}{2\lambda_{n/2}}
\right)
- 
\frac{\tilde K + \frac{4}{\gamma}}{\tilde K \P\{A_{n/2}\}}
 \frac{C_1}{C_2}
-
\frac{2K\gamma}{\tilde K}
\right]^{-1}\\
&\leq&\frac{2\left(\frac{2}{\gamma}+1\right)}{p}
\left[\frac{g(\de,\al_0)}{2}
- 
\frac{\tilde K + \frac{4}{\gamma}}{\tilde K p}
 \frac{C_1}{C_2}
\right]^{-1}\\
&\leq&\frac{8\left(\frac{2}{\gamma}+1\right)}{p g(\de,\al_0)}\\
&\leq&\frac{8}{p g(\de,\al_0)}\left(\frac{16}{ g(\de,\al_0)}+1\right)
\end{eqnarray*}
if 
\[C_2\geq \frac{4\tilde K + \frac{16}{\gamma}}{\tilde K p}
 \frac{C_1}{g(\de,\al_0)}.\]
\end{proof}%
\noindent Finally, Theorem \ref{main} follows from the previous lemma that gives the following inequality:
\begin{eqnarray*}
\frac{Q(s,n)}{K}&\leq&\inf_{m\in\Mcol_n}\left\{ \|s-s_m\|^2 + \frac{\lambda_n}{Kn} D_m\right\}\\
&\leq&\|s_{m_0(n)}-s\|^2 + 
\frac{ \lambda_{n}}{Kn}D_{m_0(n)}\\
&\leq& C_2
 \left(\frac{\lambda_{n}}{n}\right)^{\frac{2\alpha}{2\alpha+1}}.
\end{eqnarray*}
\subsection{Proofs of Theorem~\ref{main2} and Proposition~\ref{prop:approxmod}}
Theorem~\ref{main} implies that for any
 $s\in MS(\hat s_{\hat  m},\rho_\al)$,
\begin{align*}
  \sup_n\left\{\rho_{n,\al}^{-2}Q(s,n)\right\} <\infty
\end{align*}
or equivalently there exists $C >0$ such  that for any $n$,
\begin{align}
\label{eq:basic}
\inf_{m\in\Mcol_n} \left\{\|s_m-s\|^2 + \frac{\lambda_n}{n} D_m\right\}
    \leq C \rho_{n,\alpha}^2.
\end{align}
By definition of $V_n$, any function $s_m$ with $m\in\Mcol_n$ belongs
to $V_n$ and thus Inequality~\eqref{eq:basic} implies
\begin{align}\label{eq:basicrevL}
  \|P_{V_n}s-s\|^2 \leq C \rho_{n,\alpha}^2
\end{align}
that is $s\in\mathcal{L}^{\alpha}_{V}$.
By definition, $\Mcols$ is a larger collection than $\Mcol_n$ and thus
Inequality~\eqref{eq:basic} also implies that for any $n$,
\[\inf_{m\in\Mcols} \left\{\|s_m-s\|^2 + \frac{\lambda_n}{n} D_m\right\}
    \leq C \rho_{n,\alpha}^2,
\]
which turns out to be a characterization of
$\mathcal{A}^{\alpha}_{\Mcols}$
when $\rho_{n,\alpha}=\left(\frac{\la_n}{n}\right)^{\frac{\alpha}{2\alpha+1}}$ as a
consequence of the following lemma.
\begin{Lemma}
Under Assumptions of Theorem~\ref{main2},
\begin{align}\label{eq:basicrevA}
\sup_n\left\{\left(\frac{\la_n}{n}\right)^{-\frac{2\alpha}{2\alpha+1}}\inf_{m\in\Mcols} \left\{\|s_m-s\|^2 + \frac{\lambda_n}{n} D_m\right\}\right\}<\infty \Leftrightarrow s \in \mathcal{A}^{\alpha}_{\Mcols}.
\end{align}
\end{Lemma}
\begin{proof}
We denote 
\[
\tilde m(n)
=\arg\min_{m\in \Mcols }\left\{\|s_{m}-s\|^2 + \frac{\la_n}{n}
  D_m\right\}.
\]
First, let us assume that for any $n$
\[
\|s_{\tilde m(n)}-s\|^2 + \frac{\la_n}{n} D_{\tilde m(n)} \leq C_1
  \left(\frac{\la_n}{n}\right)^{\frac{2\al}{2\al+1}}
\]
where $C_1$ is a constant. Then, 
\[
D_{\tilde m(n)}\leq C_1
\left(\frac{\la_n}{n}\right)^{-\frac{1}{1+2\alpha}}.
\]
Using $\la_n\leq \la_{2n}\leq 2\la_n$,
 for  $M\in\Ne^*$, as soon as $M\geq C_1 \left(\la_1\right)^{-\frac{1}{1+2\alpha}}$,
there
exists $n\in\Ne^*$ such that
\begin{equation}
C_1\left(\frac{\la_n}{n}\right)^{-\frac{1}{1+2\alpha}}\leq
M< C_1 \left(\frac{\la_{2n}}{2n}\right)^{-\frac{1}{1+2\alpha}}
\leq C_1 2^{\frac{1}{1+2\alpha}} \left(\frac{\la_n}{n}\right)^{-\frac{1}{1+2\alpha}}.
\label{smineq}
\end{equation}
Then,
\begin{align*}
\inf_{\{m\in \Mcols :\: D_m\leq
  M\}}\|s_m-s\|^2&\leq \inf_{\left\{m \in \Mcols
:\: D_m\leq M\right\}}\left\{\|s_m-s\|^2+
\frac{\la_n}{n} D_{m}\right\}\\
&\leq\inf_{\left\{m\in \Mcols
:
\:\ D_m\leq
    C_1\left(\frac{\la_n}{n}\right)^{-\frac{1}{1+2\alpha}}\right\} }\left\{\|s_m-s\|^2+
\frac{\la_n}{n} D_{m}\right\}\\
&\leq C_1
  \left(\frac{\la_n}{n} \right)^{\frac{2\al}{1+2\alpha}}\\
&\leq C_1^{2\alpha+1} 2^{\frac{2 \al}{1+2\alpha}} M^{-2\alpha}.
\end{align*}
Conversely, assume that there exists $\tilde C_1$ satisfying
\[
\inf_{\{m\in \Mcols:\: D_m\leq M\}}\|s_m-s\|^2\leq \tilde C_1
M^{-2\al}.
\]
Then for any $T>0$,
\begin{eqnarray*}
\inf_{m\in \Mcols}\left\{\|s_m-s\|^2+ T^2
  D_m\right\}&=&\inf_{M\in\Ne^*}\inf_{\{m\in
  \Mcols:\:D_m=M\}}\left\{\|s_m-s\|^2+T^2 M\right\}\\
&\leq&\inf_{M\in\Ne^*}\left\{\tilde C_1 M^{-2\al}+ T^2 M\right\}\\
&\leq&\inf_{x\in\R^*_+}\left\{\tilde C_1 x^{-2\al}+ T^2 (x+1)\right\}\\
&\leq&                   \tilde C_1\left(\frac{T^2}{2\alpha
\tilde    C_1}\right)^{\frac{2\al}{1+2\alpha}}
+T^2 \left(\left(\frac{T^2}{2\alpha
\tilde C_1}
\right)^{-\frac{1}{1+2\al}}+1\right)\\
&\leq&C_1 \left(T^2 \right)^\frac{2\alpha}{1+2\alpha},
\end{eqnarray*}
where $C_1$ is a constant. 
\end{proof}

\noindent We have proved so far that $MS(\hat s_{\hat
  m},\rho_\al) \subset \mathcal{L}^{\alpha}_{V} \cap
\mathcal{A}^{\alpha}_{\Mcols}$. It remains to prove the converse
inclusion. Corollary~\ref{corollaire} and the previous lemma imply that 
it suffices to prove that
inequalities~\eqref{eq:basicrevL} and \eqref{eq:basicrevA} imply
inequality~\eqref{eq:basic} (possibly with a different constant $C$).

\noindent Let $s\in\mathcal{L}^{\alpha}_{V} \cap
\mathcal{A}^{\alpha}_{\Mcols}$. By inequality~\eqref{eq:basicrevA}, for every $n$,
there exists a model $m\in\Mcols$ such that
\begin{align*}
\|s_m-s\|^2 + \frac{\lambda_n}{n} D_m
    \leq C \rho_{n,\alpha}^2.
  \end{align*}
  By definition of $\Mcols$, there exists $k$ such that $m \in \Mcol'_k$.

\noindent If $k\leq n$ then $m \in \Mcol_n$ and thus
\begin{align*}
\inf_{m\in\Mcol_n}  \left\{\|s_m-s\|^2 + \frac{\lambda_n}{n} D_m\right\}
    \leq C \rho_{n,\alpha}^2.
\end{align*}

\noindent Otherwise $k>n$ and let $m'\in \Mcol$ be the model such that
$\mathcal{I}_{m}=\mathcal{I}_{m'} \cap \mathcal{J}_k$
as defined in Section~\ref{basis}. We define $m''\in\Mcol_n$ by
its index set $\mathcal{I}_{m''}=\mathcal{I}_{m'} \cap \mathcal{J}_n$.  
Remark that $m''\subset m$ and  $s_m - s_{m''}\in V_n^{\perp}$, so
\begin{align*}
\|s_{m''}-s\|^2 
+ \frac{\lambda_n}{n} D_{m''}
&=
 \|s_{m''}-s_{m}\|^2 + \|s_{m}-s\|^2  
+ \frac{\lambda_n}{n} D_{m''}\\
& \leq 
 \|P_{V_n}s - s\|^2
+ \|s_{m}-s\|^2  
+ \frac{\lambda_n}{n} D_{m}
\\
& \leq C \rho_{n,\alpha}^2.
\end{align*}
Theorem \ref{main2} is proved.

\noindent The proof of Proposition~\ref{prop:approxmod} relies on the definition
of $\widetilde{\pen}_n(m)$. Recall that
for any model $m \in \Mcol_n'$ there is a model $\tilde m\in\Mcol$
such that
\[m=\mbox{span}\left\{\p_i:\quad i\in \mathcal{I}_{\tilde m}\cap \mathcal{J}_n\right\}\]
and that \begin{align*}
 \widetilde{\pen}_n(m) = \frac{\lambda_n}{n} D_{\tilde m}.
\end{align*}
One deduces
\[
\|s_{m}-s\|^2 + \widetilde{\pen}_n(m)
=
\|s_{m}-s\|^2 + \frac{\lambda_n}{n} D_{\tilde m} \geq
\|s_{\tilde m}-s\|^2 + \frac{\lambda_n}{n} D_{\tilde m}
\]
and thus
\begin{align*}
 \inf_{m\in\mathcal{M}_n} \left\{\|s_{m}-s\|^2 +  \widetilde{\pen}_n(m)\right\} \leq C \rho_{n,\alpha}^2
\implies
 \inf_{m\in\Mcol}\left\{
\|s_{m}-s\|^2 + \frac{\lambda_n}{n} D_{m}\right\} \leq C \rho_{n,\alpha}^2.
\end{align*}
Mimicking the proof of Theorem~\ref{main2}, one obtains  Proposition~\ref{prop:approxmod}.

\subsection{Proof of Proposition \ref{nestedtheo}}
In the same spirit as in the proof of Theorem~\ref{main}, for any $n$, we denote
\begin{equation}\label{m0nest}
m_0(n)=\arg\min_{m\in\mathcal{ M}}\left\{\|s_{m}-s\|^2 + \frac{\pen(m)}{4}\right\}=\arg\min_{m\in\mathcal{ M}}\left\{\|s_{m}-s\|^2 + \frac{\la_nD_m}{4n}\right\}. 
\end{equation}
(we have set $K=4$)
and
\begin{equation}\label{mnest}
\hat m(n)=\arg\min_{m\in\mathcal{ M}}\left\{-\|\hat {s}_{m}\|^2 +
  \pen(m)\right\}=\arg\min_{m\in\mathcal{ M}}\left\{-\|\hat s_{m}\|^2 + \frac{\la_nD_m}{n}\right\}. 
\end{equation}
In the nested case, Lemma~\ref{lem:lemmastart} becomes the following
much stronger lemma:
\begin{Lemma}\label{initnest}
For any $n$, almost surely
\begin{equation}\label{equanest}\|s_{m_0(n)}-s\|^2  \leq \|\hat s_{\hat m(n)}-s\|^2.
\end{equation}
\end{Lemma} 
\begin{proof}
As the models are embedded, either $\hat{m}(n) \subset m_0(n)$ or
$m_0(n) \subset  \hat{m}(n) $.

\noindent In the first case,
$\|s_{m_0(n)}-s\|^2  \leq \|s_{\hat m(n)}-s\|^2\leq \|\hat s_{\hat
  m(n)}-s\|^2$ and thus (\ref{equanest}) holds.

\noindent Otherwise, by construction
\begin{align*}
&\begin{cases}
\|s_{m_0(n)}-s\|^2 + \frac{\la_nD_{m_0(n)}}{4n} \leq \|s_{\hat{m}(n)}-s\|^2 + \frac{\la_nD_{\hat{m}(n)}}{4n}\\
-\|\hat s_{\hat{m}(n)}\|^2 + \frac{\la_nD_{\hat{m}(n)}}{n} \leq -\|\hat s_{m_0(n)}\|^2 + \frac{\la_nD_{m_0(n)}}{n}
\end{cases}
\intertext{and thus as $m_0(n) \subset \hat{m}(n)$}
&\begin{cases}
\| s_{\hat{m}(n) \setminus m_0(n)}\|^2 \leq
\frac{\la_nD_{\hat{m}(n)}}{4n} - \frac{\la_nD_{m_0(n)}}{4n} \\
\frac{\la_nD_{\hat{m}(n)}}{n} - \frac{\la_nD_{m_0(n)}}{n}\leq \|\hat
s_{\hat{m}(n)\setminus m_0(n)}\|^2.
\end{cases}\quad
\end{align*}
Combining these two inequalities yields
\begin{align*}
  \| s_{\hat{m}(n) \setminus m_0(n)}\|^2 &\leq \frac{1}{4} \|\hat
s_{\hat{m}(n)\setminus m_0(n)}\|^2 \\
 & \leq \frac{1}{2} \left( \|\hat s_{\hat{m}(n)\setminus m_0(n)} - s_{\hat{m}(n)\setminus m_0(n)}\|^2
 + \| s_{\hat{m}(n) \setminus m_0(n)}\|^2
\right)
\intertext{and thus}
\| s_{\hat{m}(n) \setminus m_0(n)}\|^2 & \leq \|\hat s_{\hat{m}(n)\setminus m_0(n)} - s_{\hat{m}(n)\setminus m_0(n)}\|^2.
\end{align*}
Now, (\ref{equanest}) holds as
\begin{align*}
  \|s_{m_0(n)}-s\|^2 & = \|s_{\hat m(n)}-s\|^2 + \| s_{\hat{m}(n)
      \setminus m_0(n)}\|^2\\
&\leq \|s_{\hat m(n)}-s\|^2 +  \|\hat s_{\hat{m}(n)\setminus m_0(n)}
- s_{\hat{m}(n)\setminus m_0(n)}\|^2\\
 &\leq \|s_{\hat m(n)}-s\|^2 +  \|\hat s_{\hat{m}(n)}
- s_{\hat{m}(n)}\|^2 = \|\hat s_{\hat m(n)}-s\|^2.
\end{align*}
\end{proof}

\noindent Now we can conclude the proof of Proposition~\ref{nestedtheo} with an induction similar to the one used in the proof of
Lemma~\ref{induction}. Indeed, let 
\begin{align*}
  A = \|s_{m_0(n)}-s\|^2 + \frac{\la_nD_{m_0(n)}}{4n} ,
\end{align*}
\begin{align*}
A  & \leq   \|
  s_{m_0(n/2)}-s\|^2 + \frac{\la_nD_{m_0(n/2)}}{4n}\\
& \leq \beta_n \E(\| \hat s_{\hat m(n/2)}-s\|^2) + (1-\beta_n)  \|
  s_{m_0(n/2)}-s\|^2 +
  \frac{\lambda_{n}}{2\lambda_{n/2}}\frac{\la_{n/2}D_{m_0(n/2)}}{4(n/2)}.
\intertext{The choice $\beta_n = 1 -
  \frac{\lambda_{n}}{2\lambda_{n/2}}$ is such that $\delta \leq
  \beta_n \leq \frac{1}{2}$ and it implies}
A
 & \leq
 \beta_n \E(\| \hat s_{\hat m(n/2)}-s\|^2) + (1-\beta_n)  \left( \|
  s_{m_0(n/2)}-s\|^2 + \frac{\la_{n/2}D_{m_0(n/2)}}{4(n/2)} \right).
\intertext{Using now almost the same induction as in Theorem~\ref{main}, we obtain}
A
& \leq \beta_n C_1^2
  \left(\frac{2\lambda_{n/2}}{n}\right)^{\frac{2\al}{1+2\alpha}} + (1-\beta_n) C_2
  \left(\frac{2\lambda_{n/2}}{n}\right)^{\frac{2\al}{1+2\alpha}} \\
 & \leq 
  \left(\frac{2\lambda_{n/2}}{\lambda_n}\right)^{\frac{2\al}{1+2\alpha}}
(C_1^2 \beta_n C_2^{-1} + (1 -\beta_n))
C_2
  \left(\frac{\lambda_{n}}{n}\right)^{\frac{2\al}{1+2\alpha}}.
\end{align*}
where $C_1$ is a constant. It suffices thus to verify that \[\left(\frac{2\lambda_{n/2}}{\lambda_n}\right)^{\frac{2\al}{1+2\alpha}}
\left(C_1^2 \beta_n C_2^{-1} + (1 -\beta_n)\right)
\leq 1,\] which is the case as soon as $C_2 \geq \frac{C_1^2}{2g(\de,\al)}$.
\subsection{Space embeddings}
\label{sec:space-embeddings}
In this paragraph we provide many embedding properties between the functional spaces considered in Section \ref{part3}. Let us recall the following definitions:
\begin{eqnarray*}
\mathcal{B}^\alpha_{p,\infty}&=&\left\{s \in \mathbb{L}_2([0,1]):\quad \sup_{J\in \mathbb{N}}2^{J(\alpha-\frac{1}{p}+\frac{1}{2})p}\sum_{k=0}^{2^j-1}|\beta_{jk}|^p<\infty\right\};\\
\mathcal{B}^{\frac{\alpha}{1+2\alpha}}_{2,\infty}&=&\left\{s \in \mathbb{L}_2([0,1]): \quad \sup_{J\in \mathbb{N}}2^{\frac{2 J \alpha}{1+2\alpha}}\sum_{j\geq J}\sum_{k=0}^{2^j-1}\beta_{jk}^2<\infty\right\};\\
\mathcal{A}^{\alpha}_{_{\Mstratcol}}&=&\left\{s \in \mathbb{L}_2([0,1]): \quad \sup_{J \in \mathbb{N}}2^{2J \alpha}\sum_{j\geq J} \sum_{k= \lfloor2^{J}(j-J+1)^{-\theta}\rfloor}^{2^j} |\beta_{j}|_{(k)}^2<\infty\right\};\\
\mathcal{W}_{\frac{2}{1+2\alpha}}&=&\left\{s \in \mathbb{L}_2([0,1]): \quad \sup_{u>0}u^{\frac{2}{1+2\alpha}}\sum_{j =0}^\infty\sum_{k=0}^{2^j-1}\mathbf{1}_{_{|\beta_{jk}|>u}}<\infty\right\}.\\
\end{eqnarray*}

\subsubsection{Space embeddings : part $I$}

\begin{eqnarray*}
\bigcup_{p\geq 1, p>\frac{2}{1+2\alpha}}\mathcal{B}^\alpha_{p,\infty}
\overset{(i)}{\subsetneq}  \mathcal{A}^{\alpha}_{_{\Mstratcol}}\overset{(ii)}{\subsetneq} 
\mathcal{W}_{\frac{2}{1+2\alpha}}.
\end{eqnarray*}


\noindent {\bf{Proof of $(i)$.}} \\
Let $s$ belong to $B^\alpha_{p,\infty}$ with $p\geq 1$ and
$p > \frac{2}{1+2\alpha}$ and, for any scale $j\in \mathbb{N}$, let us  denote by $\left(|\beta_j|_{(k)}\right)_k$ the sequence of the non-decreasing reordered wavelet coefficients of any level $j$. Then there exists a non negative constant $C$ such that for any $j \in \mathbb{N}$
\[
\sum_{k=1}^{2^j} |\beta_{j}|_{(k)}^p \leq C 2^{-jp(\alpha+1/2-1/p)}.
\]

\noindent Fix $J \in \mathbb{N}$. If $p<2$, according to Lemma~4.16 of \cite{mas}, for all $j$ larger than $J$
\begin{align*}
\sum_{k=\lfloor2^{J}(j-J+1)^{-\theta}\rfloor+1}^{2^j}
|\beta_{j}|_{(k)}^2 &\leq C^{2/p} \ 2^{-2j(\alpha+1/2-1/p)}
\left(\lfloor2^{J}(j-J+1)^{-\theta}\rfloor\right)^{1-2/p}\\
& \leq C^{2/p} \ 2^{-2J\alpha} 2^{-2(j-J)(\alpha+1/2-1/p)}(j-J+1)^{\theta(2/p-1)}.
\end{align*}
Summing over the indices $j$ larger than $J$ yields
\begin{align*}
\sum_{j\geq J}
\sum_{k=\lfloor2^{J}(j-J+1)^{-\theta}\rfloor}^{2^j}
|\beta_{j}|_{(k)}^2 & \leq C^{2/p} 2^{-2J\alpha} \sum_{j'\geq 0} 2^{-2j'(\alpha+1/2-1/p)}(j'+1)^{\theta(2/p-1)}
\end{align*}
and thus
\begin{align*}
\sup_{J \geq 0} 2^{2J\alpha}\sum_{j \geq J} \sum_{k=
  \lfloor2^{J}(j-J+1)^{-\theta}\rfloor}^{2^j} |\beta_{j}|_{(k)}^2
\leq C^{2/p} \sum_{j'\geq 0}
2^{-2j'(\alpha+1/2-1/p)}(j'+1)^{\theta(2/p-1)} < \infty
.\end{align*} 
So $s$ belongs to $\mathcal{A}^\alpha_{{\Mstratcol}}$.\\

\noindent For the case $p=2$, 
\begin{align*}
\sum_{j\geq J}
\sum_{k = \lfloor2^{J}(j-J+1)^{-\theta}\rfloor}^{2^j}
|\beta_{j}|_{(k)}^2 & \leq
\sum_{j \geq J}
\sum_{k= 1}^{2^j}
|\beta_{j}|_{(k)}^2 \leq \sum_{j \geq J} C 2^{-2j\alpha}
\leq C \frac{2^{-2J\alpha}}{1-2^{-2\alpha}}. 
\end{align*}
Thus 
$$\sup_{J \in \mathbb{N}}2^{2J \alpha}\sum_{j\geq J} \sum_{k= \lfloor2^{J}(j-J+1)^{-\theta}\rfloor}^{2^j} |\beta_{j}|_{(k)}^2<\infty.$$
\noindent So $s$ also belongs to $\mathcal{A}^\alpha_{{\Mstratcol}}$.\\

\noindent We conclude that for any $p\geq 1$ satisfying $p > \frac{2}{1+2\alpha}$, \: $B^{\alpha}_{p,\infty} \subseteq
\mathcal{A}^\alpha_{{\Mstratcol}}.$\\
\noindent Let us now prove the strict inclusion by considering the function $s_0$ defined as follows:
\[
s_0 = \sum_{j\geq 0} \sum_{k=0}^{2^j-1}\beta_{jk}\psi_{jk}=\sum_{j\geq 0}2^{-\sqrt{j}} \psi_{j,0}.
\]

\noindent For any $(\alpha',p)$ such that $\alpha'>\max(\frac{1}{p}-\frac{1}{2},0)$
\begin{align*}
  2^{(\alpha'-\frac{1}{p}+\frac{1}{2})p j} \sum_{k=0}^{2^j-1} |\beta_{j,k} |^p = 2^{(\alpha'-\frac{1}{p}+\frac{1}{2})p j}2^{-\sqrt{j}p}
\end{align*}
and thus goes to $+\infty$ when $j$ goes to $+\infty.$ It implies
that $s_0$ does not belong to $\mathcal{B}^{\alpha}_{p,\infty}$ for any $p>\frac{2}{1+2\alpha}$.\\

\noindent Now for any $J \in \mathbb{N}$,
\begin{align*}
 2^{2J\alpha}\sum_{j \geq J} \sum_{k\geq
  \lfloor2^{J}(j-J+1)^{-\theta}\rfloor}^{2^j} |\beta_{j}|_{(k)}^2
& =  2^{2J\alpha}
\sum_{ j \geq \min \{ j' \geq J: 2^J (j'-J+1)^{-\theta} <1\}}
 2^{-2\sqrt{j}}\\
& \leq 2^{2J\alpha}
\sum_{ j \geq 2^{J/\theta}+J}
 2^{-2\sqrt{j}},
\end{align*}
which implies
\[
\sup_{J\geq 0} 2^{2J\alpha}\sum_{j \geq J} \sum_{k\geq
  \lfloor2^{J}(j-J+1)^{-\theta}\rfloor}^{2^j} |\beta_{j}|_{(k)}^2 < \infty
\]
 and thus
$s_0 \in \mathcal{A}^\alpha_{{\Mstratcol}}$. Hence $(i)$ is proved. \hfill
  $\blacksquare$\par\noindent \\

\noindent {\bf{Proof of $(ii)$.}} \\
 There is no doubt that $\mathcal{A}^{\alpha}_{_{\Mstratcol}}\subseteq
\mathcal{W}_{\frac{2}{1+2\alpha}}$ since $\mathcal{W}_{\frac{2}{1+2\alpha}}=\mathcal{A}^{\alpha}_{_{\Mwealthcol}}.$ The strict inclusion is a direct consequence of  $(i\nu ),$ just below.\hfill
  $\blacksquare$\par\noindent \\

\subsubsection{Space embeddings : part $II$}

\begin{eqnarray*}
\bigcup_{p\geq\max(1,\frac{2}{(1+2\alpha)^{-1}+2\alpha})}\mathcal{B}^\alpha_{p,\infty}
&\overset{(iii)}{\subseteq} & \mathcal{B}^{\frac{\alpha}{1+2\alpha}}_{2,\infty}\cap \mathcal{A}^{\alpha}_{_{\Mstratcol}}  \overset{(i\nu)}{\subsetneq}  \mathcal{B}^{\frac{\alpha}{1+2\alpha}}_{2,\infty}\cap\mathcal{W}_{\frac{2}{1+2\alpha}}.
\end{eqnarray*}

\noindent {\bf{Proof of $(iii)$.}} \\
Let $\alpha>0$ and  $p\geq 1$ satisfying $p\geq 2((1+2\alpha)^{-1}+2\alpha)^{-1}$.  Using the classical Besov embeddings $\mathcal{B}^{\alpha}_{p,\infty}\subseteq \mathcal{B}^{\frac{\alpha}{1+2\alpha}}_{2,\infty}$, and, according to $(i)$, we have $\mathcal{B}^{\alpha}_{p,\infty} \subsetneq \mathcal{A}^{\alpha}_{_{\Mstratcol}}$. Hence $\mathcal{B}^{\alpha}_{p,\infty} \subseteq \mathcal{B}^{\frac{\alpha}{1+2\alpha}}_{2,\infty} \cap \mathcal{A}^{\alpha}_{_{\Mstratcol}}$ and $(iii)$ is proved. \hfill
  $\blacksquare$\par\noindent \\

\noindent {\bf{Proof of $(i\nu)$.}} \\
We already know that $\mathcal{B}^{\frac{\alpha}{1+2\alpha}}_{2,\infty} \cap \mathcal{A}^{\alpha}_{_{\Mstratcol}}\subseteq
\mathcal{B}^{\frac{\alpha}{1+2\alpha}}_{2,\infty} \cap
\mathcal{W}_{\frac{2}{1+2\alpha}}.$ The strict inclusion is a direct
consequence of  $(\nu i)$ proved in the next subsection.\hfill
  $\blacksquare$\par\noindent \\

\subsubsection{A non-embedded case}

\begin{eqnarray*}
\mathcal{A}^{\alpha}_{_{\Mstratcol}}  &\overset{(\nu)}{\not\subset}& \mathcal{B}^{\frac{\alpha}{1+2\alpha}}_{2,\infty}\cap\mathcal{W}_{\frac{2}{1+2\alpha}} \quad \: \hbox{and} \quad \:
\mathcal{B}^{\frac{\alpha}{1+2\alpha}}_{2,\infty}\cap\mathcal{W}_{\frac{2}{1+2\alpha}}\overset{(\nu i)}{\not\subset}\mathcal{A}^{\alpha}_{_{\Mstratcol}}.\\
\end{eqnarray*}

\noindent {\bf{Proof of $(\nu)$.}} \\
Let us consider the function $s_0 \in \mathcal{A}^{\alpha}_{_{\Mstratcol}}$ defined in the proof of $(i)$. We already know that it does not belong to $\mathcal{B}^{\alpha'}_{p,\infty}$ for any $(\alpha',p)$ satisfying $\alpha'>\max(\frac{1}{p}-\frac{1}{2},0).$ As a consequence for the case $(\alpha',p)=(\frac{\alpha}{1+2\alpha},2)$ 
where $\alpha>0$, we deduce that $s_0$ does not belong to $\mathcal{B}^{\frac{\alpha}{1+2\alpha}}_{2,\infty}.$ \\ 
Moreover, we immediately deduce that  $\mathcal{A}^{\alpha}_{_{\Mstratcol}}  \not\subset  \mathcal{B}^{\frac{\alpha}{1+2\alpha}}_{2,\infty}\cap\mathcal{W}_{\frac{2}{1+2\alpha}}$.\hfill $\blacksquare$\par\noindent \\

\noindent {\bf{Proof of $(\nu i)$.}} \\
Let $s_1\in \mathbb{L}^2([0,1])$ whose wavelet expansion is given by
\[
s_1 = \sum_{j=0}^\infty \sum_{k=0}^{2^j-1} \beta_{jk} \psi_{jk}.
\]
We set
\begin{align*}
  \beta_{jk} = \begin{cases}
 2^{- \frac{j}{2}} & \text{if $k<2^{\frac{j}{1+2\alpha}}$}\\
0 & \text{otherwise}.
\end{cases}
\end{align*}
We are going to prove that $s_1\in \mathcal{B}^{\frac{\alpha}{1+2\alpha}}_{2,\infty}
\cap \mathcal{W}_{\frac{2}{1+2\alpha}}$ while $s_1\notin
\mathcal{A}^\alpha_{\Mstratcol}.$

\noindent Summing at a given scale $j$ yields
\begin{align*}
  \sum_{k=0}^{2^j-1} \beta_{jk}^2 = 2^{\frac{j}{1+2\alpha}} 2^{-j} 
  = 2^{-\frac{2 \alpha j}{1+2\alpha}}
  \end{align*}
and thus $s_1 \in B^{\frac{\alpha}{1+2\alpha}}_{2,\infty}.$

\noindent Let $0<u<1$ and $j_u$ the real number such that $2^{j_u}=u^{-2}.$ Then

\begin{eqnarray*}
u^{\frac{2}{1+2\alpha}}\sum_{j= 0}^\infty\sum_{k=0}^{2^j-1}\mathbf{1}_{_{|\beta_{jk}|>u}}&=&u^{\frac{2}{1+2\alpha}}\sum_{j<j_u}\sum_{k=0}^{2^j-1}\mathbf{1}_{_{|\beta_{jk}|>u}}\\
&=& u^{\frac{2}{1+2\alpha}}\sum_{j<j_u}2^{\frac{j}{1+2\alpha}}\\
&\leq&2^{\frac{1}{1+2\alpha}}(2^{\frac{1}{1+2\alpha}}-1)^{-1}.
\end{eqnarray*}
\noindent So $$\sup_{u>0}u^{\frac{2}{1+2\alpha}}\sum_{j= 0}^\infty\sum_{k=0}^{2^j-1}\mathbf{1}_{_{|\beta_{jk}|>u}}<\infty$$ and $s_1 \in \mathcal{W}_{\frac{2}{1+2\alpha}}$.\\

\noindent Let us now prove that $s_1$ does not belong to $\mathcal{A}^\alpha_{\Mstratcol}.$  Fix $J \in \mathbb{N}$ large enough. Then
\begin{eqnarray*}
E_J & = &\sum_{j\geq J} \sum_{k=
  \lfloor2^{J}(j-J+1)^{-\theta}\rfloor}^{2^j-1} |\beta_{j}|_{(k)}^2\\
  &=& \sum_{ j \geq J}
  \max\left(0,2^{j/(2\alpha+1)}-\frac{2^J}{(j-J+1)^\theta}\right) 2^{-j}. 
\end{eqnarray*}
Let $J^\star$ be the real number such that
$2^{\frac{J^\star}{1+2\alpha}}=\frac{2^J}{(J^\star-J+1)^\theta}$.\\
From $ J^\star = (2\alpha+1) J - (2\alpha+1)\theta\log_2(J^\star-J+1)$
one deduces thus
$J^\star  \leq (2\alpha+1) J$,
which implies 
  $J^\star  \geq (2\alpha+1) J - (2\alpha+1)\theta\log_2(2\alpha J +1)$,
and finally
$J^\star \leq (2\alpha+1) J - (2\alpha+1)\theta\log_2(2\alpha J +1
  - (2\alpha+1)\theta\log_2(2\alpha J +1)).$ So,
\begin{align*}
  E_ J &= \sum_{j>J^\star}    \left(2^{j/(2\alpha+1)}-\frac{2^J}{(j-J+1)^\theta}\right) 2^{-j} \\
& \geq \sum_{j>J^\star}    \left(2^{j/(2\alpha+1)}-2^{J^\star/(2\alpha+1)}\right)
2^{-j} \\
&\geq C \ 2^{-2 J^\star \alpha/(2\alpha+1)}\\
& \geq C \ (\log)^{2\alpha \theta } \ 2^{-2J\alpha}. 
\end{align*}
So,
\[
\sup_{J\geq 0} 2^{2J\alpha}\sum_{j \geq J} \sum_{k\geq
  \lfloor2^{J}(j-J+1)^{-\theta}\rfloor}^{2^j-1} |\beta_{j}|_{(k)}^2 = \infty.
\]
This implies that $s_1\notin \mathcal{A}^\alpha_{\Mstratcol}$.
Finally $(\nu i)$ is proved.
\hfill
  $\blacksquare$\par\noindent
\section*{Acknowledgments} We warmly thanks the anonymous referees for their carefull reading and 
their remarks which allow us to improve the paper.
\bibliographystyle{plain}

\end{document}